\documentclass[11pt]{article}

\usepackage{amsmath,amsthm,amssymb,tikz,bbm,float,subcaption,lineno}

\DeclareCaptionSubType*{figure}

\newtheorem*{theorem*}{Theorem}
\newtheorem{theorem}{Theorem}[section]
\newtheorem{corollary}[theorem]{Corollary}
\newtheorem{lemma}[theorem]{Lemma}
\newtheorem{proposition}[theorem]{Proposition}

\theoremstyle{definition}

\newtheorem{question}[theorem]{Question}
\newtheorem{example}[theorem]{Example}
\newtheorem{note}[theorem]{Note}

\title{Fixed-point-free fusion automorphisms}

\author{Andrew Schopieray}

\date{}

\begin{document}

\maketitle

\begin{abstract}
This is a study of fusion ring automorphisms leaving only the trivial element fixed.  We prove that a variety of classical results on fixed-point-free automorphisms of finite groups are true in the generality of fusion rings.  As a result, we show there are $8$ Grothendieck equivalence classes of fusion categories of rank less than $9$ with a fixed-point-free fusion automorphism of prime order, generalizing existing results about modular fusion categories of odd dimension.
\end{abstract}

\section{Introduction}

\par The study of fusion rings and categories in the $21$\textsuperscript{st} century is a continuation of the study of finite groups in the $20$\textsuperscript{th} century.  Every finite group $G$ can be realized as the fusion ring $\mathbb{Z}G$, its integral group ring, and as the fusion category (over $\mathbb{C}$) $\mathrm{Vec}_G^\omega$ for any $\omega\in H^3(G,\mathbb{C}^\times)$.  But not every fusion ring is isomorphic to $\mathbb{Z}G$ for some $G$, and not every fusion ring can be realized as the Grothendieck ring of a fusion category.  So one of the most pervasive questions is where the boundaries between these mathematical objects lie.

\par For example, subgroups of a finite group $G$ correspond to fusion subrings of $\mathbb{Z}G$, and thus fusion subcategories of $\mathrm{Vec}_G^\omega$.  Frobenius-Perron dimension ($\mathrm{FPdim}$) \cite[Section 3.3]{tcat} is the notion analogous to the order of a group for fusion rings and categories, and one might believe that results about orders of finite groups hold true in this more abstract setting.  The classical result is that the order of a finite group is divisible by the order of any of its subgroups, and while it is true that the Frobenius-Perron dimension of a fusion category is divisible by the Frobenius-Perron dimension of any of its fusion subcategories \cite[Theorem 7.17.6]{tcat}, the same is not true for fusion rings in general (Example \ref{oddex}).  Even more, the obvious generalization of the Sylow theorems for finite groups in this analogy fails for both fusion rings and fusion categories, unless additional restrictions are imposed (e.g.\ \cite{MR2037711}).

\par For a second example, solvability of finite groups has an analog for fusion categories \cite[Definition 9.8.1]{tcat} so that the fusion categories $\mathrm{Vec}_G^\omega$ are solvable if and only if $G$ is solvable \cite[Proposition 9.8.5(ii)]{tcat}.  With this definition, many results about solvable groups also apply to fusion categories such as the fact that any fusion category of Frobenius-Perron dimension $p^aq^b$ for primes $p,q\in\mathbb{Z}_{\geq2}$ and nonnegative integers $a,b$ is solvable \cite[Theorem 9.15.9]{tcat}.  With the existence of additional structures such as braidings, other classes of fusion categories of odd Frobenius-Perron dimension have been shown to be solvable \cite{MR4019323,MR3163515}.  It has been conjectured that all fusion categories of odd Frobenius-Perron dimension are solvable \cite{MR4516198}, and at least one notion of solvability has been proposed recently on the level of fusion rings and their generalizations \cite{MR4355908}.

\par Continuing in this thread, \emph{fixed-point-free} automorphisms of finite groups have been studied for over a century.  Of course, ``fixed-point-free'' is a misnomer since the multiplicative identity will be fixed by definition, but this term is firmly established in the group theory literature.  One of the first observations about fixed-point-free automorphisms of finite groups was if $G$ is a finite group and $\phi$ a fixed-point-free automorphism of $G$ of order 2, then $G$ is abelian of odd order and $\phi(g)=g^{-1}$ for all $g\in G$.  In a result attributed to W.\ Burnside \cite{MR0069818}, it is noted that if $G$ is a finite group with a fixed-point-free automorphism of order 3, then $G$ is nilpotent of nilpotency class at most 2.  More generally, in the doctoral thesis of J.\ G.\ Thompson \cite{MR2611480}, it is proven that any finite group with a fixed-point-free automorphism of prime order is nilpotent.  If one utilizes the classification of finite simple groups, there are very condensed proofs of other related facts about fixed-point-free automorphisms of finite groups \cite[Theorem 1.48]{MR0231903}\cite{MR1334233}\cite{MR1943704}.

\par In this manuscript we study how results about finite groups and their fixed-point-free automorphisms generalize to fusion rings and categories.  There is currently very little written about fusion rings with assumptions about their automorphisms that does not fall under the purview of finite group theory, representations of Hopf algebras, or fusion rings related affine Lie algebras and their categorifications (e.g.\ \cite{MR1887583,MR4079742}); we will attempt to recount the relevant pre-existing results here.  Recall that for commutative fusion rings $(R,B)$, the duality antiautomorphism $x\mapsto x^\ast$ for all $x\in B$ is furthermore a fusion automorphism of $R$.  There are several pre-existing results on fusion categories for which this duality is fixed-point-free.  In particular, it is known that if the Frobenius-Perron dimension of a fusion category $\mathcal{C}$ is odd, then $X\not\cong X^\ast$ for all simple objects of $\mathcal{C}$ \cite[Corollary 8.2]{MR2313527}.  When $\mathcal{C}$ is commutative, this gives a fixed-point-free automorphism of the underlying fusion ring.  This statement is false for commutative fusion rings in general with the smallest example having rank $4$ (Example \ref{oddex}).  Note that there is no reason to believe such an automorphism will correspond to a tensor autoequivalence of the fusion category in general.  Conversely, it was observed \cite[Theorem 2.2]{MR2659210} that modular fusion categories $\mathcal{C}$ such that $X\not\cong X^\ast$ for all simple $X$ in $\mathcal{C}$ are integral, i.e.\ $\mathrm{FPdim}(X)\in\mathbb{Z}$ for all $X$ in $\mathcal{C}$, and therefore $\mathrm{FPdim}(\mathcal{C})$ is an odd integer.  We prove that this statement is true for arbitrary fusion rings with fixed-point-free automorphisms of prime order (Theorem \ref{thm1}(1)), which is to say if $(R,B)$ is a fusion ring with a fixed-point free automorphism of prime order $p\in\mathbb{Z}_{\geq2}$, then $R$ is integral, and thus $\mathrm{FPdim}(R)\equiv1\pmod{p}$.  This follows from the fact that fixed-point-free automorphisms of fusion rings act without fixed-points on their isomorphism classes of irreducible representations distinct from $\mathrm{FPdim}$ (Proposition \ref{jan}).

\par Using these general results, we then analyze several special cases.  First we prove that any fixed-point-free involution (order 2) of a commutative fusion ring must be the duality homomorphism (Corollary \ref{cor:com}).  This leaves the lingering and interesting question of whether the existence of a fixed-point-free automorphism of a fusion ring of order $2$ implies commutativity as is the case for finite groups.  Next, we describe all fusion rings with fixed-point-free automorphisms $\phi$ of prime order such that $R$ has $2$ or $3$ $\phi$-orbits of basis elements.  In the former case (Section \ref{subsec:two}), all examples are integral group rings $\mathbb{Z}G$ where $G\cong C_3$ is the cyclic group of order 3, or $G\cong C_2^n$ is an elementary abelian $2$-group such that $2^n-1$ is a Mersenne Prime.  In particular, for any odd prime $p\in\mathbb{Z}_{\geq3}$ such that $p+1$ is not a power of $2$, there does not exist a fusion ring of rank $p+1$ with a fixed-point-free automorphism of any prime order.  When there are exactly $3$ $\phi$-orbits (Section \ref{subsec:three}), $R$ has a nontrivial fusion subring $\mathbb{Z}G$ where $G$ is one of the groups from the previous case.  This follows from the fact that an integral fusion ring whose basis elements have at most $3$ distinct Frobenius-Perron dimensions must have a nontrivial pointed subring (Lemma \ref{gcdlem}).

\par Lastly, we demonstrate there are exactly $10$ fusion rings (Figure \ref{fig:A}) of rank less than $9$ with fixed-point-free automorphisms of prime order $p\in\mathbb{Z}_{\geq2}$ and integer formal codegrees, in the sense of \cite[Section 2]{codegrees}.  Integer formal codegrees are necessary for a fusion ring to be categorifiable, but it is not clear if integrality of formal codegrees is a consequence of the existence of a fixed-point-free automorphism.  Eight of the $10$ classified fusion rings possess categorifications and $3$ do not; six of these $10$ rings are not isomorphic to $\mathbb{Z}G$ for a finite group $G$, $3$ of which are the character rings of the nonabelian groups of orders $3\cdot7$, $3\cdot13$, and $5\cdot11$.  This is in contrast to more constrained results that state that all modular fusion categories of rank at most $11$ \cite[Theorem 4.5]{MR2869100} and subsequently rank at most $15$ \cite[Theorem 6.3(a)]{MR4516198} are pointed.  Many questions are motivated from this initial study and deserve further analysis; we prioritize a few in the following list.

\begin{question}
Are there finitely many fusion rings of a given rank admitting a fixed-point-free automorphism of prime order?
\end{question}

\begin{question}\label{q:dual}
Does there exist a fusion ring (necessarily noncommutative) with a fixed-point-free automorphism of order $2$ which is not duality?
\end{question}

\begin{question}
Is there a nontrivial invertible basis element in every fusion ring with a fixed-point-free automorphism of prime order?
\end{question}

\begin{question}
Does there exist a noncommutative fusion ring with a fixed-point-free automorphism of prime order whose rank is less than 64?
\end{question}



\section{Basic definitions and results}

\par A fusion ring $(R,B)$ is an associative unital ring $R$ which is free as a $\mathbb{Z}$-module with a distinguished basis $B:=\{b_0=1_R,b_1,\ldots,b_n\}$ and anti-involution $b_i\mapsto b_i^\ast=:b_{i^\ast}$ of $B$ (extended linearly to all of $R$) such that $b_ib_j=\sum_{j=0}^nc_{ij}^kb_k$ for some $c_{ij}^k\in\mathbb{Z}_{\geq0}$ (the \emph{fusion rules} of $R$) and $c_{ij}^0=\delta_{ij^\ast}$.  The order of $B$ is denoted $\mathrm{rank}(R)$.  Perhaps the most vital symmetry of a fusion ring is the cyclic invariance of the fusion rules \cite[Proposition 3.1.6]{tcat} which states that for all $x,y,z\in B$,
\begin{equation}\label{cyclicinv}
c_{x,y}^{z^\ast}=c_{y,z}^{x^\ast}=c_{z,x}^{y^\ast}.
\end{equation}

\par By a fusion subring $(R',B')\subset(R,B)$, we mean $(R',B')$ is a fusion ring in its own right and $B'\subset B$.  We define the group $\mathrm{Aut}(R,B)$ as the set of permutations $\phi$ of $B$ such that $\phi$, extended linearly to $R$, is a unital ring automorphism.  Of course, there exist many unital ring automorphisms of $R$ which do not belong to $\mathrm{Aut}(R,B)$ in general.  If $\phi\in\mathrm{Aut}(R,B)$, then $\phi(x)^\ast=\phi(x^\ast)$ for all $x\in B$.   Indeed, applying $\phi$ to $xx^\ast$ yields $c_{\phi(x),\phi(x^\ast)}^{1_R}=1$.  One can replace most results about $\mathrm{Aut}(R,B)$ with analogous  results about \emph{anti}automorphisms, i.e.\ replacing the condition $\phi(xy)=\phi(x)\phi(y)$ with $\phi(xy)=\phi(y)\phi(x)$.  We will only do so briefly in Section \ref{sectiontoo} in the context of fixed-point-free involutions,. i.e.\ $\phi\circ\phi=\mathrm{id}_R$.

\begin{example}
There are myriad examples of fixed-point-free automorphisms of prime order of the integral group rings $\mathbb{Z}G$ for finite groups $G$.  It is a well-known exercise to show that there exists a fixed-point-free automorphism of order 2 of a finite group $G$ if and only if $G$ is abelian of odd order and the automorphism has $g\mapsto g^{-1}$ for all $g\in G$. 

\par When the order of the fixed-point-free automorphism is an odd prime $p\in\mathbb{Z}_{\geq3}$, there exist non-abelian examples $G$ which we will list by \texttt{GAP} identification numbers $[|G|,\#]$.  The smallest non-abelian examples have order $64$, and all such groups were classified in \cite{MR1143086}.  In particular, there are $7$ isomorphism classes of finite groups of order $64$ which admit a fixed-point-free automorphism of prime order $p$; three of these are non-abelian groups: group $[64,82]$ is an example for $p=7$, and both $[64,242]$ and $[64,245]$ are examples for $p=3$.  We note that the order of $\mathrm{Aut}(G)$ for each of these groups is over $9000$.
\end{example}

\par Let $(R,B)$ be a fusion ring.  Let $\mathrm{Irr}(R)$ be the set of irreducible complex representations of $R$ up to isomorphism, i.e.\ unital algebra homomorphisms $\varphi:R\to\mathrm{End}(V)$ for some finite-dimensional complex vector space $V$.  The only irreducible representation existing for all fusion rings $(R,B)$ is the \emph{Frobenius-Perron} representation $\mathrm{FPdim}:R\to\mathbb{R}_{\geq1}$ which is the unique unital ring homomorphism $\rho:R\to\mathbb{C}$ such that $\rho(x)>0$ for all $x\in B$ \cite[Proposition 3.3.6(3)]{tcat}.  Returning to the motivation of the definition of $\mathrm{Aut}(R,B)$, the requirement that $\phi\in\mathrm{Aut}(R,B)$, versus $\phi$ being an arbitrary unital ring automorphism of $R$, is mild but ensures that for all $x\in R$, $\mathrm{FPdim}(\phi(x))=\mathrm{FPdim}(x)$ \cite[Proposition 3.3.13]{tcat}.  If $\mathrm{FPdim}(x)=1$ for some $x\in R$, then $x\in B$, $xx^\ast=1_R$, and we call $x$ \emph{invertible}.  The $\mathbb{Z}$-linear span of all invertible elements of $R$ form a fusion subring $R_\mathrm{pt}$ of $R$, which is characteristic in $R$, in the sense that it is preserved under all $\phi\in\mathrm{Aut}(R,B)$.  It is clear that the invertible objects of $R$ form a finite group $G$, and $R=R_\mathrm{pt}$ if and only if $R\cong\mathbb{Z}G$.  In general, the group $G$ acts on $R$ by left or right multiplication, and for each $x\in B$, the stabilizer subgroup $G_x\subset G$ is the subset of $g\in G$ such that $c_{x,x^\ast}^g=c_{x,x^\ast}^{g^{-1}}=c_{g,x}^x=1$ by the cyclic invariance of the fusion rules (Equation (\ref{cyclicinv})) and the fact that $xx^\ast$ is self-dual. 

\par The finite group $\mathrm{Aut}(R,B)$ acts on $\mathrm{Irr}(R)$; for all $\varphi\in\mathrm{Irr}(R)$ and $\phi\in\mathrm{Aut}(R,B)$, $\varphi\circ \phi\in\mathrm{Irr}(R)$ which may or may not be distinct from $\varphi$.  The characters $\chi_\varphi:=\mathrm{Tr}(\varphi):R\to\mathbb{C}$ of $\varphi\in\mathrm{Irr}(R)$ satisfy the familiar orthogonality relation \cite[Proposition 2.12]{MR2535395} which is if $\varphi\not\cong\rho\in\mathrm{Irr}(R)$,
\begin{equation}\label{orthochar}
\sum_{x\in B}\chi_\varphi(x)\chi_\rho(x^\ast)=0.
\end{equation}
When $\rho=\mathrm{FPdim}$, or in general when $\chi_\rho$ is real-valued, one can ignore the duality in this formula, and we will do so in the future without note.

\begin{proposition}\label{jan}
Let $(R,B)$ be a fusion ring and $\phi\in\mathrm{Aut}(R,B)$ of prime order $p\in\mathbb{Z}_{\geq2}$.  If $1_R$ is the unique fixed-point of the action of $\phi$ on $B$, then $\mathrm{FPdim}$ is the unique fixed-point of the action of $\phi$ on $\mathrm{Irr}(R)$.
\end{proposition}

\begin{proof}
Let $\phi\in\mathrm{Aut}(R,B)$ be fixed-point-free of prime order $p\in\mathbb{Z}_{\geq2}$.  Assume to the contrary that $\varphi$ is a fixed-point of the action of $\phi$ on $\mathrm{Irr}(R)$, which is to say that $\varphi\circ\phi^j=\varphi$, and thus $\chi_{\varphi\circ\phi^j}=\chi_\varphi$, for all $1\leq j\leq p$.  In particular, $\chi_\varphi(x)=\chi_{\varphi}(\phi^j(x))$ for all $1\leq j\leq p$ and $x\in B$.  Let $\Gamma$ be a set of representatives of nontrivial $\phi$-orbits in $B$, which all have order $p$ by the assumption that $p$ is prime.  Then applying the orthogonality relation of Equation (\ref{orthochar}) to $\chi_\varphi$ and $\mathrm{FPdim}$,
\begin{equation}
0=\sum_{x\in B}\chi_\varphi(x)\mathrm{FPdim}(x^\ast)=1+p\sum_{x\in\Gamma}\chi_\varphi(x)\mathrm{FPdim}(x),
\end{equation}
which erroneously implies $1/p$ is an algebraic integer.  So we must conclude $\mathrm{FPdim}$ is the unique fixed-point of $\phi$ acting on $\mathrm{Irr}(R)$.
\end{proof}

A coarser, but more easily computable, invariant of a fusion ring $(R,B)$ is its multi-set of formal codegrees, indexed by irreducible representations of $R$.  For each irreducible $\varphi:R\to\mathrm{End}(V)$, the \emph{formal codegree} $f_\varphi\in\mathbb{C}$ of $\varphi$ can be defined in terms of the central element $\alpha:=\sum_{b\in B}bb^\ast$ which satisfies  $\varphi(\alpha)=\dim(\varphi)f_{\varphi}\mathrm{id}_V$ \cite[Lemma 2.6]{codegrees}.  Hence the multi-set of the formal codegrees $f_\varphi$ for $\varphi\in\mathrm{Irr}(R)$ are the eigenvalues of (left) multiplication by $\alpha$ occurring with multiplicity $\dim(\varphi)^2$ \cite[Remark 2.11]{ost15} and so are algebraic integers closed under Galois conjugacy.  Furthermore the Frobenius-Perron dimension $\mathrm{FPdim}(R):=f_{\mathrm{FPdim}}=\sum_{x\in B}\mathrm{FPdim}(x)^2$ is one of the formal codegrees of $R$.  These numerical invariants are heavily controlled by the constraint \cite[Proposition 2.10]{ost15}
\begin{equation}\label{codegrees}
\sum_{\varphi\in\mathrm{Irr}(R)}\dfrac{\dim(\varphi)}{f_\varphi}=1,
\end{equation}
and the fact that they are all real numbers greater than or equal to $1$ \cite[Remark 2.12]{ost15}.  It is clear that for any $\phi\in\mathrm{Aut}(R,B)$, $\phi(\alpha)=\alpha$, and therefore $f_{\varphi}=f_{\varphi\circ\phi}$.  The following theorem summarizes the above results for fixed-point-free automorphisms of prime order.

\begin{theorem}\label{thm1}
Let $(R,B)$ be a fusion ring and $\phi\in\mathrm{Aut}(R,B)$ of prime order $p\in\mathbb{Z}_{\geq2}$.  If $\phi$ is fixed-point-free, then
\begin{enumerate}
\item $\mathrm{FPdim}(x)\in\mathbb{Z}$ for all $x\in B$ and $\mathrm{FPdim}(R)\equiv1\pmod{p}$;
\item $\mathrm{rank}(R)$ and the number of $\varphi\in\mathrm{Irr}(R)$ with $\dim(\varphi)=1$ are congruent to 1 modulo $p$; 
\item the number of $\varphi\in\mathrm{Irr}(R)$ with $\dim(\varphi)=k$ for any $k\in\mathbb{Z}_{\geq2}$ is divisible by $p$;
\item the formal codegrees of the $\phi$-orbit of $\varphi\in\mathrm{Irr}(R)$ are equal.
\end{enumerate}
\end{theorem}

\begin{proof}
Recall that $\mathrm{FPdim}(x)$ are algebraic integers for all $x\in B$ as eigenvalues of nonnegative integer matrices.  Proposition \ref{jan} implies that $\mathrm{FPdim}$ is the unique character $\chi$ of $R$ such that $\chi(x)=\chi(\phi(x))$ for all $x\in B$.  But for any $\sigma\in\mathrm{Gal}(\overline{\mathbb{Q}}/\mathbb{Q})$ where $\overline{\mathbb{Q}}$ is the algebraic closure of $\mathbb{Q}$, the character $\sigma(\mathrm{FPdim})$ has this property since $\mathrm{FPdim}$ is preserved by $\phi$, thus $\sigma(\mathrm{FPdim})=\mathrm{FPdim}$.  Therefore $\mathrm{FPdim}(x)$ are rational integers for all $x\in B$.  Moreover $\mathrm{FPdim}(R)\equiv1\pmod{p}$ by collecting $x\in B$ by $\phi$-orbits.  The second and third results follow as any nontrivial $\phi$-orbit in $\mathrm{Irr}(R)$ has order $p$, and the final claim is true for any $\phi\in\mathrm{Aut}(R,B)$\end{proof}

\begin{corollary}\label{cor:2}
Let $(R,B)$ be a commutative fusion ring.  If $2$ divides $\mathrm{FPdim}(R)$, then there exists nontrivial $x\in B$ with $x=x^\ast$.
\end{corollary}

\begin{proof}
Assume for all nontrivial $x\in B$, $x\neq x^\ast$.  Then duality is a fixed-point-free automorphism of order $2$ since $R$ is commutative.  The contrapositive of Theorem \ref{thm1}(1) then implies $2\nmid\mathrm{FPdim}(R)$.
\end{proof}

\begin{example}\label{oddex}
The converse of Corollary \ref{cor:2} is false.   Consider the commutative rank 4 fusion ring $(R,B)$ with fusion rules
\begin{align}
\left[\begin{array}{cccc}
1&0&0&0 \\
0&1&0&0 \\
0&0&1&0 \\
0&0&0&1
\end{array}\right]\qquad
\left[\begin{array}{cccc}
0&1&0&0 \\
1&0&0&0 \\
0&0&1&0 \\
0&0&0&1
\end{array}\right] \\
\left[\begin{array}{cccc}
0&0&1&0 \\
0&0&1&0 \\
1&1&1&0 \\
0&0&0&2
\end{array}\right]\qquad
\left[\begin{array}{cccc}
0&0&0&1 \\
0&0&0&1 \\
0&0&0&2 \\
1&1&2&1 
\end{array}\right]
\end{align}
The Frobenius-Perron dimensions of $x\in B$ are $1,1,2,3$ and $\mathrm{FPdim}(R)=1^2+1^2+2^2+3^2=3\cdot5$.  But every basis element is self-dual.
\end{example}

So we see the relationship between the parity of Frobenius-Perron dimensions and duality in fusion rings is subtle.  For example, there exist fusion rings whose Frobenius-Perron dimensions of nontrivial basis elements are all odd, but the Frobenius-Perron dimension of the ring is even, e.g.\ $\mathbb{Z}G$ for a group of even order.  But there cannot exist a fusion ring whose Frobenius-Perron dimensions of nontrivial basis elements are all even, and the Frobenius-Perron dimension of the ring is odd.  Indeed, taking the Frobenius-Perron dimension of $xx^\ast$ for any nontrivial basis element $x$ in such a ring will lead to a contradiction.  This reasoning provides a ``local'' version of Corollary \ref{cor:2} which does not require commutativity.

\begin{proposition}\label{evenprop}
Let $(R,B)$ be a fusion ring.  If $2$ divides $\mathrm{FPdim}(x)^2$ for some $x\in B$, then there exists nontrivial $y\in B$ with $y=y^\ast$.
\end{proposition}

\begin{proof}
Assume to the contrary, that $y\neq y^\ast$ for all nontrivial $y\in B$.  Since $xx^\ast$ is self-dual, then $c_{x,x^\ast}^y=c_{x,x^\ast}^{y^\ast}$ for all $y\in B$.  Let $\Gamma\subset B$ be representatives of the nontrivial orbits of $B$ under the duality antiautomorphism.  We compute
\begin{align}
\mathrm{FPdim}(x)^2&=\mathrm{FPdim}(xx^\ast)\\
&=1+2\sum_{y\in\Gamma}c_{x,x^\ast}^y\mathrm{FPdim}(y)
\end{align}
and therefore upon rearrangement,
\begin{align}
-\dfrac{1}{2}&=\dfrac{-\mathrm{FPdim}(x)^2}{2}+\sum_{y\in\Gamma}c_{x,x^\ast}^y\mathrm{FPdim}(y).\label{n}
\end{align}
The right-hand side of Equation (\ref{n}) is an algebraic integer since $2$ divides $\mathrm{FPdim}(x)^2$, but the left-hand side is not. 
\end{proof}
The relationship between fixed-point-free automorphisms of order $2$, parity of Frobenius-Perron dimensions, and duality is discussed in greater detail in Section \ref{sectiontoo}.

Another consequence of Theorem \ref{thm1}(3) is that the rank of a noncommutative fusion ring with a fixed-point-free automorphism of prime order $p\in\mathbb{Z}_{\geq2}$ cannot be too small relative to $p$.  To achieve a reasonable lower bound for our purposes requires the following elementary lemma.

\begin{lemma}\label{lem:meh}
There is no fusion ring $(R,B)$ such that $\mathrm{FPdim}(R)\in\mathbb{Z}_{\geq1}$ has multiplicity $1$ as a formal codegree, and $R$ has at exactly $2$ distinct formal codegrees.
\end{lemma}

\begin{proof}
Assume that $R$ has exactly 2 distinct formal codegrees, $f:=\mathrm{FPdim}(R)\in\mathbb{Z}_{\geq1}$ and $g\in\mathbb{Z}_{\geq1}$, where integrality follows from the formal codegrees being closed under Galois conjugacy.  Then with $h$ being the sum of dimensions of irreducible representations of $R$ distinct from $\mathrm{FPdim}$, we have $1/f+h/g=1$ by Equation (\ref{codegrees}).  Therefore $h/g\leq(g-1)/g$.  If $g<f$, $1/f+h/g<1$.  Moreover no such fusion ring exists.
\end{proof}

\begin{proposition}\label{comcor}
Let $(R,B)$ be a fusion ring and $\phi\in\mathrm{Aut}(R,B)$ be fixed-point-free of prime order.  If $R$ is noncommutative, then $\mathrm{rank}(R)\geq11$. 
\end{proposition}

\begin{proof}
This proof is a brief computation by Theorem \ref{thm1}.  In particular if $r:=\mathrm{rank}(R)$, we seek to create a list of dimensions of irreducible representations $d_1,\ldots,d_\ell$, not all equal to $1$, such that $r=\sum_{j=1}^\ell d_j^2$.  Specifically, let $n_k\in\mathbb{Z}_{\geq0}$ be the number of $\varphi\in\mathrm{Irr}(R)$ with $\dim(\varphi)=k\in\mathbb{Z}_{\geq1}$.  There must exist a prime $p\in\mathbb{Z}_{\geq2}$ such that $r\equiv1\pmod{p}$ and $n_k\equiv0\pmod{p}$.  All $d_j$, $p$, $k$, and $n_k$ are bounded by $r$, so this is a finite computation which can be done by hand.  One feasible solution exists with $r<11$ which is the case $p=2$ and $r=9$ with $d_1=1$, $d_2=d_3=2$.  But this violates Lemma \ref{lem:meh} since the two-dimensional irreducible representations share the same formal codegree by Theorem \ref{thm1}(4).
\end{proof}

Our last result of the section notes that for fusion rings of small rank, which are predominately commutative, many of the formal codegrees of $R$ are determined by the group of invertible elements of $R$, since the action of of the invertible elements on the noninvertible basis elements is often trivial.

\begin{lemma}\label{formlem}
Let $(R,B)$ be a commutative fusion ring with $R\neq R_\mathrm{pt}$ such that $G$ the group of invertible objects of $R$ acts trivially on the set of noninvertible basis elements.  Then $R$ has $|G|$ as a formal codegree of multiplicity exactly $|G|-1$.
\end{lemma}

\begin{proof}
Let $\varphi\in\mathrm{Irr}(G)$ be nontrivial.  We can extend $\varphi$ to a ring homomorphism $\tilde{\varphi}:R\to\mathbb{C}$ by assigning $\tilde{\varphi}(x)=0$ for all noninvertible $x\in B$.  Note that for noninvertible $x,y\in B$, $xy$ either contains all invertible elements as summands in the case $y=x^\ast$, or none, by the cyclic invariance of the fusion rules (Equation (\ref{cyclicinv})).  Therefore the only nontrivial verification that $\tilde{\varphi}$ is a well-defined homomorphism is
\begin{equation}\label{ate}
\tilde{\varphi}(xx^\ast)=\sum_{g\in G}\tilde{\varphi}(g)+\sum_{y\in B\setminus G}\tilde{\varphi}(y)=\sum_{g\in G}\varphi(g),
\end{equation}
vanishes, which follows since $\varphi\in\mathrm{Irr}(G)$ is nontrivial.  Conversely, if $\varphi:R\to\mathbb{C}$ is a ring homomorphism and $\varphi(x)=0$ for all noninvertible basis elements $x$, then the left-hand side of Equation (\ref{ate}) vanishes, thus $\varphi|_G\in\mathrm{Irr}(G)$ is nontrivial.  Now if $\varphi:R\to\mathbb{C}$ is any ring homomorphism, then $f_\varphi\geq|G|$ by \cite[Example 2.4]{codegrees} since $\varphi(x)$ is a root of unity for all invertible $x\in B$, with equality if and only if $\varphi(x)=0$ for all noninvertible $x\in B$.  Moreover $|G|$ is a formal codegree of multiplicity exactly $|G|-1$.
\end{proof}



\section{When $p=2$}\label{sectiontoo}

\par In the case when $p$ is the unique even prime, more can be said about fixed-point-free $\phi\in\mathrm{Aut}(R,B)$.  For example, when $R\cong\mathbb{Z}G$ for a finite group $G$, then the existence of such $\phi$ implies $G$ is an abelian group of odd order, and $\phi$ is the inversion map $g\mapsto g^{-1}$.

\subsection{Relation with duality}

\begin{proposition}\label{propcom}
Let $(R,B)$ be a fusion ring and $\phi$ be an antiautomorphism of $R$ of order 2.  Then $\phi$ is fixed-point-free if and only if $\phi(x)=x^\ast\neq x$ for all $x\in B$.
\end{proposition}

\begin{proof}
The converse half of the statement is trivial, so assume $\phi$ is fixed-point-free.  For all $x\in B$, $x\phi(x)$ is $\phi$-invariant, and by definition $\phi(x)=x^\ast$ if and only if $c^{1_R}_{x,\phi(x)}=1$.  Let $\Gamma_0\subset B$ be a set of representatives of $\phi$-orbits of $B$ which appear as summands of $x\phi(x)$.  If $\phi(x)\neq x^\ast$, then since $\phi$ preserves Frobenius-Perron dimensions \cite[Proposition 3.3.9]{tcat},
\begin{equation}
\mathrm{FPdim}(x)^2=\mathrm{FPdim}(x\phi(x))=2\sum_{y\in\Gamma_0}c_{x,\phi(x)}^y\mathrm{FPdim}(y),
\end{equation}
hence $2\mid\mathrm{FPdim}(x)^2$ for any $x\in B$ with $\phi(x)\neq x^\ast$.  Now consider $xx^\ast$, which is a self-dual object of $R$, hence its summands are closed under duality.  Let $\Gamma_1\subset B$ be all of the summands $y$ of $xx^\ast$ for which $\phi(y)\neq y^\ast$, and let $\Gamma_2\subset B$ be a set of representatives of the $\phi$-orbits of all nontrivial summands $y$ of $xx^\ast$ for which $\phi(y)=y^\ast$.  It is evident from their definition that $\Gamma_1$ and $\Gamma_2$ are disjoint and contain all nontrivial summands of $xx^\ast$.  We compute
\begin{align}
\mathrm{FPdim}(xx^\ast)=1+\sum_{y\in\Gamma_1}c_{x,x^\ast}^y\mathrm{FPdim}(y)^2+2\sum_{z\in \Gamma_2}c_{x,x^\ast}^z\mathrm{FPdim}(z)^2,
\end{align}
and therefore upon rearrangement,
\begin{equation}
\dfrac{1}{2}=\dfrac{1}{2}\left(\mathrm{FPdim}(x)^2-\sum_{y\in\Gamma_1}c_{x,x^\ast}^y\mathrm{FPdim}(y)^2\right)-\sum_{z\in \Gamma_2}c_{x,x^\ast}^z\mathrm{FPdim}(z)^2
\end{equation}
is an algebraic integer, which is false.  Therefore $\phi(x)=x^\ast$ for all $x\in B$.
\end{proof}

\begin{corollary}\label{cor:com}
Let $(R,B)$ be a commutative fusion ring and $\phi\in\mathrm{Aut}(R,B)$ of order 2.  If $\phi$ is fixed-point-free, then $\phi(x)=x^\ast$ for all $x\in B$.
\end{corollary}

\begin{proof}
The set of antiautomorphisms and automorphisms are the same when $R$ is commutative. 
\end{proof}

\begin{proposition}\label{threethree}
Let $(R,B)$ be a fusion ring and $f\in\mathrm{Aut}(R,B)$ be fixed-point-free of order 2.  For all $x\in B$, if $x\phi(x)=\phi(x)x$, then $\phi(x)=x^\ast$.
\end{proposition}

\begin{proof}
The proof is identical to that of Proposition \ref{propcom} since the element $x\phi(x)$ is $\phi$-invariant under the assumption $x\phi(x)=\phi(x)x$.
\end{proof}

\begin{example} Basis elements of fusion rings need not commute with their duals.  If they did, then Proposition \ref{threethree} would immediately imply the only fixed-point-free involution is duality.  This still may be true but the reason must be slightly deeper.  For example, one of the fusion rings associated to the extended Haagerup subfactor \cite[Figure 8]{MR2979509} possesses a basis element (labeled ``$B$'' in loc.\ cit.) with the following fusion rule, which does not commute with its dual.

\begin{equation}
\left[\begin{array}{cccccccc}
0&0&0&0&0&0&1&0 \\
0&0&0&0&0&1&0&0 \\
0&0&0&1&1&0&0&0 \\
0&0&1&1&1&1&0&1 \\
0&1&0&1&1&1&1&0 \\
0&0&1&1&1&1&0&0 \\
0&0&0&1&0&0&0&0 \\
1&0&0&0&0&1&0&0
\end{array}\right]
\end{equation}
\end{example}


\subsection{Representation theory and formal codegrees}

\begin{proposition}\label{propcomp}
Let $(R,B)$ be a commutative fusion ring and $\phi\in\mathrm{Aut}(R,B)$ which is fixed-point-free of order 2.  For all $\varphi\in\mathrm{Irr}(R)$, if $\varphi\neq\mathrm{FPdim}$, then $\varphi(x)\not\in\mathbb{R}$ for some $x\in B$.
\end{proposition}

\begin{proof}
Commutativity of $R$ implies $\phi(x)=x^\ast\neq x$ for all $x\in B$, so if $\varphi(x)=\overline{\varphi(x)}=\varphi(x^\ast)$ for all $x\in B$, then $\varphi\neq\mathrm{FPdim}$ is a fixed-point of the action of $\phi$ on $\mathrm{Irr}(R)$, violating Proposition \ref{jan}.
\end{proof}

\begin{proposition}\label{prop1}
Let $(R,B)$ be a commutative fusion ring and $\phi\in\mathrm{Aut}(R,B)$ have order 2.  If $\phi$ is fixed-point-free, then $2$ does not divide $f_\varphi$ for any $\varphi\in\mathrm{Irr}(R)$.
\end{proposition}

\begin{proof}
Let $\alpha:=\sum_{x\in B}xx^\ast$.  The eigenvalues of the matrix of (left) multiplication by $\alpha$ are $f_\varphi$ for $\varphi\in\mathrm{Irr}(R)$ since $R$ is commutative, with multiplicity $1$ \cite[Remark 2.11]{ost15}.  Let $\Gamma$ be an index set of representatives of nontrivial $\phi$-orbits in $B$.  As $R$ is commutative, $\phi(x)=x^\ast$ for all $x\in B$ by Corollary \ref{cor:com}.  Then
\begin{equation}
\alpha=1+2\sum_{x\in \Gamma}xx^\ast,
\end{equation}
\par Therefore the matrix $N$ of multiplication by $\alpha$ is the $|B|\times|B|$ identity added to a matrix entirely composed of even integers.  Moreover the elements of the diagonal are odd integers.  We claim the determinant of $N$ is an odd integer as well.  The case $|B|=1$ is trivial, and we continue by induction.  If $|B|=n$ for some $n>1$, then the terms in the cofactor expansion along the first row are all even, except the first term which is odd by the inductive hypothesis.
\end{proof}

\par We expect the analogous result to Proposition \ref{prop1} to be true for odd primes $p\in\mathbb{Z}_{\geq3}$.  When $p=2$, small formal codegrees are highly restrictive.

\begin{lemma}\label{lem3}
Let $(R,B)$ be a commutative fusion ring and $\phi\in\mathrm{Aut}(R,B)$ be fixed-point-free of order 2.  If $3$ is a formal codegree of $R$, then $R_\mathrm{pt}\cong\mathbb{Z}C_3$. 
\end{lemma}

\begin{proof}
Let $\varphi\in\mathrm{Irr}(R)$ such that $f_\varphi=3$.  If $R=R_\mathrm{pt}$ we are done, so we may assume that $\varphi\neq\mathrm{FPdim}$.  We must have
\begin{equation}
[\mathbb{Q}(\varphi(y):y\in B):\mathbb{Q}]\in\{1,2\}
\end{equation}
as all Galois conjugates of $\varphi$ will share the same formal codegree which can be computed as
\begin{equation}\label{eighteen}
3=\sum_{y\in B}|\varphi(y)|^2=1+\sum_{y\in B\setminus\{1_R\}}|\varphi(y)|^2.
\end{equation}
We may conclude $\mathbb{Q}(\varphi(y):y\in B)$ is a complex quadratic field by Proposition \ref{propcomp}.  
But $|\varphi(y)|^2=\varphi(y)\overline{\varphi(y)}\in\mathbb{Z}_{\geq0}$ is the algebraic norm of $\varphi(y)\in\mathbb{Q}(\varphi(y):y\in B)$, thus there exist distinct nontrivial $x_1,x_2\in B$ such that $|\varphi(x_1)|^2=|\varphi(x_2)|^2=1$.  Therefore $\varphi(y)=0$ for all $y\not\in\{1_R,x_1,x_2\}$, and moreover $x:=x_1=x_2^\ast$.  The only complex quadratic units which are not integers are $\zeta\in\{\pm\zeta_4,\pm\zeta_3,\pm\zeta_3^2\}$, so without loss of generality, assume $\varphi(x)=\zeta$ and thus $\varphi(x^\ast)=\zeta^{-1}$.  Orthogonality of $\varphi$ with $\mathrm{FPdim}$ gives
\begin{equation}
0=1+(\zeta+\zeta^{-1})\mathrm{FPdim}(x),
\end{equation}
which implies that $\zeta=\zeta_3$, up to a relabelling of $x$ and $x^\ast$, and that $\mathrm{FPdim}(x)=\mathrm{FPdim}(x^\ast)=1$.  There can be no other nontrivial invertible basis elements $y\in B$ since necessarily $\varphi(y)$ would be nonzero, violating Equation (\ref{eighteen}).
\end{proof}

\begin{lemma}\label{thebiglemma}
Let $(R,B)$ be a commutative fusion ring with $R\neq R_\mathrm{pt}$ and $\phi\in\mathrm{Aut}(R,B)$ be fixed-point-free of order 2.  If $5$ is a formal codegree of $R$, then $5$ is a formal codegree of multiplicity $4$. 
\end{lemma}

\begin{proof}
Let $\varphi\in\mathrm{Irr}(R)$ such that $f_\varphi=5$.  By Proposition \ref{jan} and Theorem \ref{thm1}(4), the multiplicity of $5$ as a formal codegree is either $2$ or $4$ since $R\neq R_\mathrm{pt}$; any higher even multiplicity violates Equation (\ref{codegrees}).  Assume the multiplicity of $5$ as a formal codegree is $2$.  Then $\mathbb{Q}(\varphi(y):y\in B)$ is a complex quadratic field by Proposition \ref{propcomp}.  The definition of $f_\varphi$ is $5=1+\sum_{y\in B\setminus\{1_R\}}|\varphi(y)|^2$ with $|\varphi(y)|^2\in\mathbb{Z}_{\geq0}$.  One possibility is that there exists a unique (up to duality) $x\in B$ such that $|\varphi(x)|^2=|\varphi(x^\ast)|^2=2$.  In this case orthogonality with $\mathrm{FPdim}$ gives
\begin{equation}\label{eqaux}
0=1+(\varphi(x)+\varphi(x^\ast))\mathrm{FPdim}(x).
\end{equation}
Since $\varphi(x)+\varphi(x^\ast)=2\mathrm{Re}(\varphi(x))\neq0$ for this to be true, then $-1/(2\mathrm{Re}(\varphi(x)))=\mathrm{FPdim}(x)$.  But $\varphi(x)$ is a complex quadratic integer so its real part must be a half integer, which forces $\mathrm{Re}(\varphi(x))=-1/2$ so that $\mathrm{FPdim}(x)\in\mathbb{Z}_{\geq1}$.  Let $b\in\mathbb{Z}_{\neq0}$ such that $\varphi(x)=(1/2)(-1+b\sqrt{n})$ for some $n\in\mathbb{Z}_{<0}$.  Then
\begin{equation}\label{tcho}
2=|\varphi(x)|^2=(1/2)(-1+b\sqrt{n})(1/2)(-1-b\sqrt{n})=(1/4)(1-b^2n).
\end{equation}
Thus $b^2n=-7$ and furthermore $b=\pm1$ and $n=-7$.  Equation (\ref{eqaux}) then implies $\mathrm{FPdim}(x)=1$, which is a contradiction since $\varphi(x)$ is nonzero, and not a root of unity.

\par The only other possibility is that there exist unique (up to duality) nontrivial $x,y\in B$ such that $x\neq y^\ast$ and $|\varphi(x)|^2=|\varphi(y)|^2=1$.  If $\varphi(x)=a_x+b_x\sqrt{n}$ for some half-integers $a_x,b_x$ and $n\in\mathbb{Z}_{<0}$, then $1=|\varphi(x)|^2=a_x^2-b_x^2n$.  Since $-b_x^2n\geq0$, then there are only three cases:
\begin{align}
&&a_x&=0&\text{ and }&&-b_x^2n&=1 \\
\Rightarrow&& b_x&=\pm1&\text{ and }&&n&=-1, \nonumber \\ &&a_x&=\pm1/2&\text{ and }&&-b_x^2n&=3/4&& \\
\Rightarrow&& b_x&=\pm1/2&\text{ and }&&n&=-3,\text{ and} \nonumber\\
&&a_x&=\pm1&\text{ and }&&b_x&=0&&&.
\end{align}
The same is true for $\varphi(y)=a_y+b_y\sqrt{n}$ for some half integers $a_y,b_y$.  Orthogonality of $\varphi$ and $\mathrm{FPdim}$ gives
\begin{equation}\label{theabove}
-1/2=a_x\mathrm{FPdim}(x)+a_y\mathrm{FPdim}(y).
\end{equation}
If $a_x=0$, then either $a_y=0$ or $a_y=\pm1$; neither satisfies Equation (\ref{theabove}).  If $a_x=\pm1/2$, then $a_y=\pm1/2$ or $a_y=\pm1$.  In the former case we may assume $a_x=1/2$ and $a_y=-1/2$ without loss of generality.  Thus $\mathrm{FPdim}(y)-1=\mathrm{FPdim}(x)$ which cannot occur since this implies at least one of $\mathrm{FPdim}(x)$ or $\mathrm{FPdim}(y)$ is even, violating Proposition \ref{evenprop}.  In the latter case, orthogonality of $\varphi$ with $\overline{\varphi}$ gives
\begin{equation}
0=1+2\mathrm{Re}(\varphi(x)^2)+2\mathrm{Re}(\varphi(y)^2)=1-2+2=1.
\end{equation}
So we must conclude there is no viable character $\varphi$ and thus the multiplicity of $5$ as a formal codegree is not $2$, so it must be $4$.
\end{proof}

\begin{proposition}\label{propfive}
Let $(R,B)$ be a commutative fusion ring and $\phi\in\mathrm{Aut}(R,B)$ be fixed-point-free of order 2.  If $5$ is a formal codegree of $R$, then $\mathrm{R}_\mathrm{pt}\cong\mathbb{Z}C_5$. 
\end{proposition}

\begin{proof}
If $R$ is pointed, we are done.  Otherwise $R$ has $5$ as a formal codegree of multiplicity $4$ from Lemma \ref{thebiglemma}.  Let $\varphi_1,\overline{\varphi_1},\varphi_2,\overline{\varphi_2}$ be the irreducible representations of $R$ with formal codegree $5$.  Let $x\in B$ be such that $\varphi_1(x)\not\in\mathbb{R}$ which exists by Proposition \ref{propcomp}.  As in the proof of Lemma \ref{thebiglemma}, if $\varphi_1$ and $\varphi_2$ are not Galois conjugate, then $\mathbb{Q}(\varphi(x))$ is a complex quadratic field and $|\varphi(x)|^2=|\varphi(x^\ast)|^2\in\{1,2\}$.  If $|\varphi(x)|^2=|\varphi(x^\ast)|^2=2$, then $\varphi(y)=0$ for all $y\in B\setminus\{1_R,x,x^\ast\}$.  Moreover a contradiction is reached in the exact manner as in the discussion around Equation (\ref{tcho}).  Thus $\varphi_1(x)$ is a complex quadratic unit, i.e.\ $\varphi_1(x)\in\{\pm\zeta_4,\pm\zeta_3,\pm\zeta_3^2\}$.  By the same reasoning, there must exist a unique up to duality $y\in B\setminus\{1_R,x,x^\ast\}$ such that $\varphi(y)\in\{\pm\zeta_4,\pm\zeta_3,\pm\zeta_3^2\}$ as well.  Orthogonality of $\varphi_1$ and $\overline{\varphi_1}$ insists $-1/2=\mathrm{Re}(\varphi_1(x)^2)+\mathrm{Re}(\varphi_1(y)^2)$ and $\mathbb{Q}(\varphi(x))=\mathbb{Q}(\varphi(y))$.  A finite check shows there are no possible solutions, so we must conclude $\varphi_1$ and $\varphi_2$ are Galois conjugate.

\par With $\varphi_1$ and $\varphi_2$ Galois conjugate, assume first that $|\varphi_1(x)|^2$ and $|\varphi_2(x)|^2$ are less than or equal to $1$.  Then $\varphi_1(x),\varphi_2(x)$ are roots of unity \cite[Lemma 3.3.14]{tcat}, therefore $|\varphi_1(x)|^2=|\varphi_2(x)|^2=1$.  Since $1+2|\varphi_1(x)|^2<5$, there must exist $y\in B\setminus\{1_R,x,x^\ast\}$ such that $\varphi_1(y)\neq0$, and if $\varphi_1(y)$ is not a root of unity, then either $|\varphi_1(y)|^2$ or $|\varphi_2(y)|^2$ is strictly greater than 1, violating the fact that $f_{\varphi_1}=f_{\varphi_2}=5$.  So we conclude $\varphi_1(x),\varphi_1(y),\varphi_2(x),\varphi_2(y)$ are roots of unity and at least one has exactly $4$ distinct Galois conjugates, say $\varphi_1(x)$ without loss of generality.  In this case we must have $\varphi_1(y),\varphi_2(y)\subset\mathbb{Q}(\varphi_1(x))$ and $\varphi_1(z)=0$ for all $z\not\in\{1_R,x,x^\ast,y,y^\ast\}$.  The possible primitive roots of unity $\varphi_1(x)$ with exactly $4$ Galois conjugates have orders in the set $\{5,8,10,12\}$.  There are no solutions to the orthogonality relation of $\varphi_1$ and $\overline{\varphi_1}$: $-1/2=\mathrm{Re}(\varphi_1(x)^2)+\mathrm{Re}(\varphi_1(y)^2)$ when $\varphi_1(x)$ and $\varphi_1(y)$ are $8$th roots of unity, and when they are $12$th roots of unity, the only feasible solutions have $2\mathrm{Re}(\varphi_1(x))=\pm\sqrt{3}$ and $\mathrm{Re}(\varphi_1(y))=0$ without loss of generality.  Orthogonality of $\varphi_1$ and $\mathrm{FPdim}$ would then imply $0=1\pm\sqrt{3}\mathrm{FPdim}(x)$, which cannot occur since $\mathrm{FPdim}(x)\in\mathbb{Z}$.  Moreover $\varphi_1(x)$ and $\varphi_1(y)$ are $10$th roots of unity.  For all solutions to $-1/2=\mathrm{Re}(\varphi(x)^2)+\mathrm{Re}(\varphi(y)^2)$, we have $2\mathrm{Re}(\varphi_1(x))=\epsilon_x\alpha$ and $2\mathrm{Re}(\varphi_1(y))=\epsilon_y\alpha^{-1}$ where $\epsilon_x,\epsilon_y\in\{\pm1\}$ and $\alpha$ is $(1/2)(1+\sqrt{5})$.  Thus orthogonality of $\varphi_1$ and $\mathrm{FPdim}$ gives
\begin{equation}
-1=\epsilon_x\alpha\mathrm{FPdim}(x)+\epsilon_y\alpha^{-1}\mathrm{FPdim}(y).
\end{equation}
Clearly $\mathrm{FPdim}(x)=\mathrm{FPdim}(y)$, or else the right-hand side is irrational, thus $-1/(\epsilon_x\alpha+\epsilon_y\alpha^{-1})=\mathrm{FPdim}(x)=\mathrm{FPdim}(y)$.  The only selection of signs $\epsilon_x,\epsilon_y$ producing a positive integer Frobenius-Perron dimension is $\epsilon_x=\epsilon_y=-1$ with $\mathrm{FPdim}(x)=\mathrm{FPdim}(y)=1$.  And as $\varphi_1(z)=0$ for all $z\not\in\{1_R,x,x^\ast,y,y^\ast\}$, then $R_\mathrm{pt}\cong\mathbb{Z}C_5$ in this case.

\par Finally, we need to demonstrate the case $|\varphi_1(x)|^2<1$ and $|\varphi_2(x)|^2>1$ is not possible, without loss of generality.  In this case $|\varphi_1(x)|^2$, and thus $|\varphi_1(y)|^2$ are totally positive quadratic integers, say $|\varphi_1(x)|^2=a_x+b_x\sqrt{n}$ and $|\varphi_1(y)|^2=a_y+b_y\sqrt{n}$ for some half integers $a_x,a_y,b_x,b_y$ and $n\in\mathbb{Z}_{\geq1}$ subject to the formal codegree conditions for $\varphi_1$ and $\varphi_2$:
\begin{align}\label{realeq}
2&=a_x+a_y\pm(b_x+b_y)\sqrt{n}.
\end{align}
Therefore $a_x+a_y=2$, and to ensure total positivity, the real constraints are $a_x,a_y\in\{1/2,1,3/2\}$, and the complex constraints are $|b_x|<a_x/\sqrt{n}$ and $|b_y|\sqrt{n}<a_y/\sqrt{n}$.  If $a_x$ (or $a_y$) is $1/2$, there are no positive integers $n$ such that $|b_x|$ (or $|b_y|$) can be a half integer.  If $a_x$ (or $a_y$) is $1$, then $|b_x|=1/2$ (or $|b_y|=1/2$) for any positive square-free integer $n$, which forces $n\equiv1\pmod{4}$.  But for the smallest possible value $n=5$, $|b_x|<1/\sqrt{5}<1/2$.  Lastly if $a_x$ (or $a_y$) is $3/2$, then $n\equiv1\pmod{4}$, so the only possibility is that $n=5$ and $|b_x|=1/2$ (or $|b_y|=1/2$).  In conclusion, the only possible values for $a_x$ and $a_y$ are $3/2$, but this violates Equation (\ref{realeq}).
\end{proof}

\begin{example}
Proposition \ref{propfive} is false for fusion rings with self-dual basis elements.  Choose a sign $\pm$ and $a\in\mathbb{Z}$ with $a\pm1\geq0$.  Then the rank 4 commutative fusion ring with nontrivial fusion rules
\begin{equation}
\left[\begin{array}{cccc}
0 & 1 & 0 & 0 \\
1 & a & a\pm1 & a\pm1 \\
0 & a\pm1 & a\pm1 & a \\
0 & a\pm1 & a & a
\end{array}\right]
\end{equation}
\begin{equation}
\left[\begin{array}{cccc}
0 & 0 & 1 & 0 \\
0 & a\pm1 & a\pm1 & a \\
1 & a\pm1 & a & a\pm1 \\
0 & a & a\pm1 & a
\end{array}\right]
\end{equation}
\begin{equation}
\left[\begin{array}{cccc}
0 & 0 & 0 & 1 \\
0 & a\pm1 & a & a \\
0 & a & a\pm1 & a \\
1 & a & a & a
\end{array}\right]
\end{equation}
has formal codegrees $5$, $5$, and the roots of
\begin{equation}
t^2-(27a^2\pm24a+12)t+45a^2\pm40a+20.
\end{equation}
One may convince themself that this is the smallest rank for which a counterexample exists as an exercise.  Note that the formal codegrees are all integers if and only if the chosen sign is positive and $a=0$; in this case the above fusion ring is the character ring $R_{D_5}$ of the dihedral group of order $10$.  When the chosen sign is negative and $a=1$, the above fusion ring is the direct product $F\times F$ where $F$ is the rank $2$ fusion ring with nontrivial basis element $x$ such that $x^2=1_F+x$.
\end{example}


\section{Examples of fixed-point-free automorphisms of prime order}\label{exsec}

Let $(R,B)$ be a fusion ring with a fixed-point-free automorphism $\phi$ of prime order $p\in\mathbb{Z}_{\geq2}$ for the remainder of this section.  We organize our discussion around the number of $\phi$-orbits of $B$.  The unit $1_R$ is an orbit itself, so the trivial fusion ring is the only possible fusion ring with exactly $1$ $\phi$-orbit.  But the identity automorphism does not have prime order so this case is not relevant.

\subsection{Two orbits}\label{subsec:two}

Assume there is exactly one nontrivial $\phi$-orbit of basis elements, hence $\mathrm{rank}(R)=p+1$.  Then Theorem \ref{thm1}(1) implies there exists $d\in\mathbb{Z}_{\geq1}$ such that $\mathrm{FPdim}(x)\in\{1,d\}$ for all $x\in B$ and $d$ satisfies $d^2=nd+1$ for some $n\in\mathbb{Z}_{\geq0}$ by taking the Frobenius-Perron dimension of $xx^\ast$ for any nontrivial $x\in B$.  The only positive integer root of $d^2-nd-1$ is $d=1$, when $n=0$.  Therefore $R\cong\mathbb{Z}G$ for a finite group of order $p+1$ admitting a fixed-point-free automorphism of prime order $p$.  We include a proof of the following well-known lemma, which characterizes such rings, for completeness.
  
\begin{lemma}\label{lem:too}
Let $G$ be a finite group.  If $\mathrm{Aut}(G)$ acts transitively on $G\setminus\{e\}$, then $G$ is an elementary abelian $p$-group for a prime $p\in\mathbb{Z}_{\geq2}$.
\end{lemma}

\begin{proof}
Let $p\in\mathbb{Z}_{\geq2}$ be any prime dividing $|G|$.  Cauchy's theorem implies there exists an element $g\in G\setminus\{e\}$ of order $p$.  If $h\in G\setminus\{e\}$, there exists $\phi\in\mathrm{Aut}(G)$ such that $\phi(g)=h$, hence $h^p=\phi(g)^p=\phi(g^p)=\phi(e)=e$ as well.  Moreover every nontrivial element of $G$ has order $p$ and therefore $G$ is a $p$-group.  As all $p$-groups are nilpotent, $Z(G)$ is nontrival.  But $Z(G)\subset G$ is a characteristic subgroup (preserved under all automorphisms of $G$) in general, thus $G=Z(G)$ and $G$ is abelian, and moreover an elementary abelian $p$-group.
\end{proof}

If $p=2$, then $R\cong\mathbb{Z}C_3$, and if $p+1$ is even, then $2^n-1=p$ is prime for some $n\in\mathbb{Z}_{\geq2}$, i.e.\ $|G|-1$ is a \emph{Mersenne prime}.  This currently provides 51 more examples; the largest known example has order just shy of 25 million digits, and the corresponding $p$ is the largest known prime number at this point in history.


\subsection{Three orbits}\label{subsec:three}

Assume there are exactly two nontrivial $\phi$-orbits of basis elements, hence $\mathrm{rank}(R)=2p+1$.  Then Theorem \ref{thm1}(1) implies there exist $d_1,d_2\in\mathbb{Z}_{\geq1}$ such that $\mathrm{FPdim}(x)\in\{1,d_1,d_2\}$ for all $x\in B$.

\begin{lemma}\label{gcdlem}
Let $(R,B)$ be a nontrivial integral fusion ring.  If there exist $d_1,d_2\in\mathbb{Z}_{\geq1}$ such that $\mathrm{FPdim}(x)\in\{1,d_1,d_2\}$ for all $x\in B$, then $R_\mathrm{pt}$ is nontrivial.
\end{lemma}

\begin{proof}
If $R=R_\mathrm{pt}$, we are done as $R$ was assumed to be nontrivial.  So consider the case $d_1=d_2\neq1$.  For any nontrivial $x\in B$, taking the Frobenius-Perron dimension of $xx^\ast$ gives $d_1^2=1+nd_1$ for some $n\in\mathbb{Z}_{\geq0}$ which has no positive integer solutions different from $1$.  Therefore $d_1>d_2$ without loss of generality if $1_R$ is the unique invertible element.  Taking the Frobenius-Perron dimension of $xx^\ast$ for $x\in B$ with $\mathrm{FPdim}(x)=d_1$,
\begin{align}
&&d_1^2&=1+n_1d_1+n_2d_2\label{eq1} \\
\Rightarrow&&1&=(d_1-n_1)d_1+(-n_2)d_2
\end{align}
with $n_1,n_2\in\mathbb{Z}_{\geq0}$.  A solution to such a linear Diophantine equation exists if and only if $d_1,d_2$ are coprime.  Now let $x,y\in B$ with $x\neq y$ such that $\mathrm{FPdim}(x)=\mathrm{FPdim}(y)=d_2$.  Taking the Frobenius-Perron dimension of $y^\ast x$,
\begin{equation}
d_2^2=m_1d_1+m_2d_2
\end{equation}
for some $m_1,m_2\in\mathbb{Z}_{\geq0}$.  Hence $m_1d_1\equiv0\pmod{d_2}$.  If $m_1>0$, then $m_1=\ell d_2$ for some $\ell\in\mathbb{Z}_{\geq1}$ since $\gcd(d_1,d_2)=1$ which would imply $d_2\geq\ell d_1$.  As $d_2<d_1$, we must conclude $m_1=0$.  Thus $0=c_{y^\ast,x}^{z^\ast}=c_{x,z}^y$ for any $z\in B$ with $\mathrm{FPdim}(z)=d_1$.  Since $y\neq x$ was arbitrary, $xz=d_1x$ by comparing Frobenius-Perron dimensions, and moreover $c_{x,x^\ast}^z=d_1$ contradicting the fact that $d_1>d_2$.  So we must conclude $R_\mathrm{pt}$ is nontrivial.
\end{proof}

\begin{note}
Lemma \ref{gcdlem} and its proof are attributed to \cite[Proposition 3.2]{alekseyev2023classification} whose conclusion is that $R$ is not simple.  As previously observed in loc.\ cit., this is an optimal result, in the sense that there exist integral fusion rings $(R,B)$ with $R_\mathrm{pt}$ trivial, and with basis elements of 4 distinct Frobenius-Perron dimensions.  The smallest such fusion ring is the character ring of the alternating group $A_5$, with Frobenius-Perron dimensions of basis elements $1,3,3,4,5$.
\end{note}

Back to the three-orbit problem, we may now assume $d_1=1$ without loss of generality.  Assume further that $d_2\neq1$.  Then by Section \ref{subsec:two}, $R_\mathrm{pt}\cong\mathbb{Z}_{C_3}$ or $R_\mathrm{pt}\cong\mathbb{Z}C_2^n$ for some $n\in\mathbb{Z}_{\geq2}$ such that $2^n-1$ is a Mersenne prime.  In the former case, $p=2$ and $\mathrm{rank}(R)=5$, thus $R$ is commutative by Proposition \ref{comcor}.  Moreover the automorphism $\phi$ is duality by Corollary \ref{cor:com}.  Therefore the two noninvertible basis elements $x,y$  satisfy $x=y^\ast$ and are fixed by multiplication by the invertible elements, hence $d_2^2=nd_2+3$ for some $n\in\mathbb{Z}_{\geq0}$.  But $d_2\in\mathbb{Z}_{\geq2}$, hence $d_2=3$ and $n=2$.  This forces $1=c_{x,y}^x=c_{x,y}^y=c_{x,x}^x$ by the fact that $xy$ is self-dual and the cyclic symmetry of the fusion rules.  Moreover $c_{x,x}^y=2$ by considering Frobenius-Perron dimensions.  The fusion of this ring is then determined by the fusion subring $\mathbb{Z}C_3$ acting trivially on $x,y$, and the following fusion rule for $x$, since the fusion rule for $y$ is its transpose.
\begin{equation}
\left[
\begin{array}{c|cc|cc}
0 & 0 & 0 & 0 & 1 \\\hline
0 & 0 & 0 & 0 & 1 \\
0 & 0 & 0 & 0 & 1 \\\hline
1 & 1 & 1 & 1 & 1 \\
0 & 0 & 0 & 2 & 1
\end{array}
\right]
\end{equation}
One may recognize this as the character ring of $C_7\rtimes C_3$, the unique non-abelian group of order $3\cdot7$ up to isomorphism.

\par Alternatively, $R_\mathrm{pt}\cong\mathbb{Z}C_2^n$ for some $n\in\mathbb{Z}_{\geq2}$ such that $p=2^n-1$ is a Mersenne prime.  As $p$ is odd, the ring $R$ is self-dual, and the action of $R_\mathrm{pt}$ on the $p$ non-invertible objects must be trivial, i.e.\ $x^2$ contains all invertible elements as summands.  Moreover $\mathrm{FPdim}(x)$ is a root of $x^2-kx-2^n$ for some $k\in\mathbb{Z}_{\geq0}$.  This bounds the fusion coefficients of $R$ by $\mathrm{FPdim}(x)\leq2^n$, producing a very restricted set of examples for each Mersenne prime $p$.

\begin{example}\label{ex2}
Assume $n=2$, thus $p=3$ and $k\in\{0,3\}$ in the notation introduced above.  Consider first the case $k=0$.   As $\mathrm{FPdim}(x)=2$ and $x^2$ contains all invertible elements, then $0=c_{x,x}^{\phi^2(x)}=c_{\phi(x),\phi(x)}^x=c_{x,\phi(x)}^{\phi(x)}$.  Similarly $c_{x,\phi^2(x)}^{\phi^2(x)}=0$, which determines the fusion rule for $x$ (Figure \ref{fig:so532}), which in turn determines the fusion rules for $R$.  We denote this fusion ring by $S_2$.  The only remaining option is $k=3$, thus $\mathrm{FPdim}(x)=4$.  If we set $a:=c_{x,x}^x$, $b:=c_{x,x}^{\phi(x)}=c_{\phi^2(x),\phi^2(x)}^x=c_{x,\phi^2(x)}^{\phi^2(x)}$, then
\begin{equation}
3-(a+b)=c_{x,x}^{\phi^2(x)}=c_{\phi(x),\phi(x)}^x=c_{x,\phi(x)}^{\phi(x)}.
\end{equation}
Along with the fact that $c_{\phi(x),\phi(x)}^{\phi^2(x)}=c_{x,x}^{\phi(x)}$, the choice of nonnegative integers $a,b$ with $a+b\leq3$ determines the fusion rule for $x$, and the commutation relation of $x$ and $\phi(x)$ implies $a=0$ and $b=1$.   We denote this ring by $S_4$.  The fusion rules for $x$ in the cases $k=0$ and $k=3$ are given in Figure \ref{fig:g21e}.

\begin{figure}[H]
\centering
\begin{subfigure}{.5\textwidth}
  \centering
\begin{equation*}
\left[
\begin{array}{c|ccc|ccc}
0 & 0 & 0 & 0 & 1 & 0 & 0 \\\hline
0 & 0 & 0 & 0 & 1 & 0 & 0 \\
0 & 0 & 0 & 0 & 1 & 0 & 0 \\
0 & 0 & 0 & 0 & 1 & 0 & 0 \\\hline
1 & 1 & 1 & 1 & 0 & 0 & 0 \\
0 & 0 & 0 & 0 & 0 & 0 & 2 \\
0 & 0 & 0 & 0 & 0 & 2 & 0
\end{array}
\right]
\end{equation*}
  \caption{$k=0$}
  \label{fig:so532}
\end{subfigure}%
\begin{subfigure}{.5\textwidth}
  \centering
\begin{equation*}
\left[
\begin{array}{c|ccc|ccc}
0 & 0 & 0 & 0 & 1 & 0 & 0 \\\hline
0 & 0 & 0 & 0 & 1 & 0 & 0 \\
0 & 0 & 0 & 0 & 1 & 0 & 0 \\
0 & 0 & 0 & 0 & 1 & 0 & 0 \\\hline
1 & 1 & 1 & 1 & 0 & 1 & 2 \\
0 & 0 & 0 & 0 & 1 & 2 & 1 \\
0 & 0 & 0 & 0 & 2 & 1 & 1
\end{array}
\right]
\end{equation*}
  \caption{$k=3$}
  \label{fig:so52}
\end{subfigure}
\caption{Fusion rule for $x$ in the rings $S_2$ and $S_4$}
\label{fig:g21e}
\end{figure}  
\end{example}

\par Lastly, assume $R\cong\mathbb{Z}G$ is pointed.  These groups are a very small subset of a larger family (finite or infinite) known as \emph{almost homogenous} groups, i.e.\ those groups $G$ such that the action of $\mathrm{Aut}(G)$ on $G$ has at most 3 orbits.  All such groups (finite or infinite) were classified in \cite{MR1484565}.  If $|G|$ is divisible by two distinct primes, $p,q$, then $|G|=pq$ and $G$ is nonabelian with $q\equiv1\pmod{p}$ without loss of generality, or else $G$ is cyclic and has elements with 4 distinct orders, which lie in distinct $\phi$-orbits.  The Sylow theorems imply the unique $q$-Sylow subgroup $C_q$ is characteristic, hence $q=3$ and $\phi$ must be the inverse map (duality).  But $S_3\cong C_3\rtimes C_2$ has elements of order 2, so no such example exists.  The only other possibility is that $|G|=q^n$ for some $n\in\mathbb{Z}_{\geq1}$, and that $q^n=2p+1$ for some prime $q$.  In other words, $p$ is odd.  By McKay's proof of the Cauchy Theorem, the number of elements $g\in G$ of order $q$ is congruent to $-1$ modulo $q$, thus $G$ is a $q$-group of exponent $q$.  If $G$ is abelian, then $G$ is an elementary abelian $q$-group.  If $G$ is nonabelian, then $G$ is a (generalized) Heisenberg group \cite[Section 4]{MR1484565} of exponent $q$.

\par The simplest examples of such finite groups $G$ are cyclic groups of order $q=2p+1$ where $p,q$ are primes.  Primes with this relation have a long history, and with this labelling $q$ is called a \emph{safe} prime and $p$ is called a \emph{Sophie Germain} prime.  We have $\mathrm{Aut}(G)\cong C_{2p}$ and the fixed-point subgroup of any nontrivial automorphism is trivial since $q$ is prime.  There are exactly $p-1$ automorphisms $\phi$ of $C_{2p}$ of order $p$, each having 3 $\phi$-orbits.

\par It should be noted that there is no known infinite family of groups of this type due to the condition $p=(1/2)(q^n-1)$.  For $q=3$, the smallest four primes $p$ are $13,$ $1093$, $797161$, and $3754733257489862401973357979128773$.





\section{Integer formal codegrees}

\par Fusion rings with integer formal codegrees are highly restrictive by Equation (\ref{codegrees}).  In particular, there exist only finitely many lists of positive integers of a given length satisfying such an equation by the analogous classical argument bounding the number of finite groups with a given number of conjugacy classes \cite{Landau1903}.  Therefore the Frobenius-Perron dimension of such a fusion ring of fixed rank is bounded.  The result then follows since there exist only finitely many fusion rings of bounded Frobenius-Perron dimension.  One can refer to \cite[Lemma 3.14]{MR3486174} for a proof of this fact which is stated in terms of fusion categories, but uses no categorical data in the proof.  The objective of this section is to use these bounds to classify fusion rings with fixed-point-free automorphisms of prime order of rank less than $9$, which are listed in Figure \ref{fig:A}.  Note that all such fusion rings must be commutative by Proposition \ref{comcor}, so this is an inoccuous assumption moving forward.

\subsection{Bounds on $\mathrm{FPdim}$}

\par Assume $(R,B)$ is a commutative fusion ring with integer formal codegrees and a fixed-point free automorphism $\phi$ of prime order $p\in\mathbb{Z}_{\geq2}$.  There are $n:=(\mathrm{rank}(R)-1)/p$ nontrivial $\phi$-orbits of $\mathrm{Irr}(R)$ and thus positive integers $f_1,\ldots,f_n$ such that
\begin{equation}
\dfrac{1}{p\mathrm{FPdim}(R)}+\sum_{j=1}^n\dfrac{1}{f_j}=\dfrac{1}{p}.
\end{equation}  
It is known \cite{192178} that in this case, $p\mathrm{FPdim}(R)$, and moreover all $f_j$, are bounded by $A_{n+1}$, where $A_1=p$ and inductively $A_{j+1}=A_j(A_j+1)$ for $2\leq j\leq n$.  We will explain how this bounded provides a classification of fusion rings in a few of the cases already described above, to prepare for a novel classification in Section \ref{nequals3}.  For example, when $n=1$, $p\mathrm{FPdim}(\mathcal{C})\leq p(p+1)$, thus $\mathrm{FPdim}(\mathcal{C})\leq p+1=\mathrm{rank}(R)$, thus $R\cong\mathbb{Z}G$ for a finite group $G$, as we proved without the assumption of integer formal codegrees or commutativity in Section \ref{subsec:two}.

\par For $n=2$, $p\mathrm{FPdim}(\mathcal{C})\leq p(p+1)(p^2+p+1)$.  Hence $\mathrm{FPdim}(\mathcal{C})\leq(p+1)(p^2+p+1)$.  When $p=2$ this gives a bound of $21$.  The only possible formal codegrees by an exhaustive search are
\begin{align}
[\mathrm{FPdim}(R),f_1,f_2]\in\{[5,5,5],\text{ }[9,9,3],[21,7,3]\}.
\end{align}
But there must also exist positive integer Frobenius-Perron dimensions $a,b$ such that $\mathrm{FPdim}(R)=1+2a^2+2b^2$.  There is a unique solution $a=b=1$ for $\mathrm{FPdim}(R)=5$ which is realized by the fusion ring $\mathbb{Z}C_5$, a unique solution $a=1,b=3$ for $\mathrm{FPdim}(R)=21$ which is realized by the fusion ring $R_{C_7\rtimes C_3}$, and no solution for $a,b$ when $\mathrm{FPdim}(R)=9$.  When $p=3$, we have $\mathrm{FPdim}(R)\leq52$.  The only possible formal codegrees by an exhaustive search are
\begin{align}
[\mathrm{FPdim}(R),f_1,f_2]\in\{&[7,7,7],[10,10,5],[16,16,4], \\
&[28,14,4],[40,8,5],[52,13,4]\}.
\end{align}
But there must also exist positive integer Frobenius-Perron dimensions $a,b$ such that $\mathrm{FPdim}(R)=1+3a^2+3b^2$.  There is a unique solution $a=b=1$ for $\mathrm{FPdim}(R)=7$ realized by the fusion ring $\mathbb{Z}C_7$, a unique solution $a=1,b=2$ for $\mathrm{FPdim}(R)=16$ realized by the fusion ring $S_2$ and a unique solution $a=1,b=4$ for $\mathrm{FPdim}(R)=52$ realized by the fusion ring $S_4$.  There is a unique, but spurious solution for $\mathrm{FPdim}(R)=40$: $a=2$ and $b=3$, which violates Lemma \ref{gcdlem} and so there are no fusion rules with these dimensions and formal codegrees.  Again, this was known from Section \ref{subsec:three} without the assumption of integral formal codegrees or commutativity.  

\subsection{Four $\phi$-orbits; the case $n=3$}\label{nequals3}

\par In this case
\begin{equation}
\mathrm{FPdim}(\mathcal{C})\leq(p + 1) (p^2 + p + 1) (p^4 + 2 p^3 + 2 p^2 + p + 1).
\end{equation}
For $p=2$ this bound is $903$.  There are $14$ distinct lists of feasible formal codegrees $[\mathrm{FPdim}(R),f_1,f_2,f_3]$, and checking for the existence of positive integers $a,b,c$ such that $\mathrm{FPdim}(R)=1+2(a^2+b^2+c^2)$, we find $11$ distinct pairings of feasible formal codegrees and triples $a,b,c$.  Of those for which $a,b,c$ are not distinct, we can discard any solution for which one of $a,b,c$ is not $1$ by Lemma \ref{gcdlem}.  The feasible solution $[\mathrm{FPdim}(R),f_1,f_2,f_3]=[147, 3, 7, 49]$ with $[a,b,c]=[1,6,6]$ cannot occur since $\mathrm{rank}(R_\mathrm{pt})=3$ which is coprime to $2$, and there exist only basis elements of two distinct Frobenius-Perron dimensions.  Hence $\mathrm{FPdim}(x)=3$ for any noninvertible $x\in B$ by measuring the Frobenius-Perron dimension of $xx^\ast$.  The feasible solution $[\mathrm{FPdim}(R),f_1,f_2,f_3]=[315, 3, 7, 45]$ with $[a,b,c]=[2,3,12]$ is spurious since this would have a fusion subring of rank $5$ with $3$ distinct Frobenius-Perron dimensions of basis elements, which does not exist from Section \ref{exsec}.  Lastly, the feasible formal codegrees $[\mathrm{FPdim}(R),f_1,f_2,f_3]=[55,15,11,3]$ with unique dimensions $[a,b,c]=[1,1,5]$ violate Lemma \ref{formlem} as $R_\mathrm{pt}$ must act trivially on $B$ in this case.  The remaining feasible solutions are listed in Figure \ref{fig:aaa}.
\begin{figure}[H]
\centering
\begin{equation*}
\begin{array}{|cc|}
\hline \text{Formal codegrees} & \text{Basis }\mathrm{FPdims}\\\hline\hline
7, 7,7, 7,7, 7,7 		& 1,1,1,1,1,1,1 \\
39, 13,13, 13,13, 3,3 	& 1,1,1,3,3,3,3 \\
55, 11,11, 5,5, 5,5 	& 1,1,1,1,1,5,5 \\
119, 51,51, 7,7, 3,3	& 1,1,1,3,3,7,7\\\hline
\end{array}
\end{equation*}
    \caption{Feasible rank $7$ fusion ring data w.\ fixed-point-free automorphism of order $2$ and integer formal codegrees}%
    \label{fig:aaa}%
\end{figure}
The case $\mathrm{FPdim}(R)=7=\mathrm{rank}(R)$ must have $R\cong\mathbb{Z}C_7$.  For the cases which are not integral groups rings, note that $\mathrm{rank}(R_\mathrm{pt})$ is coprime to the number of noninvertible $x\in B$ of any fixed Frobenius-Perron dimension, hence $1=c_{g,x}^x=c_{x,x^\ast}^{g^{-1}}=c_{x,x^\ast}^g$ for any invertible $g\in B$ by the cylic invariance of the fusion rules (Equation \ref{cyclicinv}) and the fact that $xx^\ast$ is self-dual.

\par Assume $\mathrm{FPdim}(R)=39$ and let $x,y\in B$ be distinct such that $x\neq y^\ast$ and $\mathrm{FPdim}(x)=\mathrm{FPdim}(y)=3$.  Ordering the noninvertible basis elements $x,x^\ast,y,y^\ast$, the fusion rule for $x$ is
\begin{equation}
\left[\begin{array}{c|cc|cc|cc}
0 & 0 & 0 & 0 & 1 & 0 & 0 \\\hline
0 & 0 & 0 & 0 & 1 & 0 & 0 \\
0 & 0 & 0 & 0 & 1 & 0 & 0 \\\hline
1 & 1 & 1 & a & a & 1-a & 1-a \\
0 & 0 & 0 & b & a & 3-n & c \\\hline
0 & 0 & 0 & c & 1-a & d & 2+a-c-d \\
0 & 0 & 0 & 3-n & 1-a & 2a+b+c-d-1 & d
\end{array}\right]
\end{equation}
where $a\in\{0,1\}$, $b,c,d\in\{0,1,2,3\}$, and $n:=a+b+c$, using only the cyclic symmetry of the fusion rules and the Frobenius-Perron dimension constraints.  Commutativity of the above with its transpose determines $a=0$, and $[b,c,d]\in\{[1,0,0],[2,0,1]\}$ up to a relabeling $y\leftrightarrow y^\ast$.  The case $[b,c,d]=[1,0,0]$ is erroneous since this implies $c_{y,y}^y=c_{y,y^\ast}^{y^\ast}=c_{y,y^\ast}^y=1$ forcing $c_{y,y}^{y^\ast}=-1$ by Frobenius-Perron dimension comparison.  The fusion rules for $x$ and $y$ are then determined as
\begin{align}
\left[\begin{array}{c|cc|cc|cc}
0 & 0 & 0 & 0 & 1 & 0 & 0 \\\hline
0 & 0 & 0 & 0 & 1 & 0 & 0 \\
0 & 0 & 0 & 0 & 1 & 0 & 0 \\\hline
1 & 1 & 1 & 0 & 0 & 1 & 1 \\
0 & 0 & 0 & 2 & 0 & 1 & 0 \\\hline
0 & 0 & 0 & 0 & 1 & 1 & 1 \\
0 & 0 & 0 & 1 & 1 & 0 & 1
\end{array}\right]\qquad\left[\begin{array}{c|cc|cc|cc}
0 & 0 & 0 & 0 & 0 & 0 & 1 \\\hline
0 & 0 & 0 & 0 & 0 & 0 & 1 \\
0 & 0 & 0 & 0 & 0 & 0 & 1 \\\hline
0 & 0 & 0 & 1 & 1 & 0 & 1 \\
0 & 0 & 0 & 1 & 0 & 1 & 1 \\\hline
1 & 1 & 1 & 1 & 1 & 0 & 0 \\
0 & 0 & 0 & 0 & 1 & 2 & 0
\end{array}\right]
\end{align}
These are the fusion rules of the character ring of the unique nonabelian group of order $39$ up to isomorphism, $C_{13}\rtimes C_3$.

\par  The only unknown fusion coefficient for a fusion ring under our current assumptions with $\mathrm{FPdim}(R)=55$ is $c_{x,x}^x=c_{x,x^\ast}^{x^\ast}$ since $c_{x,x}^{x^\ast}=5-c_{x,x}^x$ and $c_{x,x^\ast}^x=4-c_{x,x}^x$.  But the fusion rule for $x$ must commute with its transpose which forces $c_{x,x}^x=2$, hence the fusion rule for noninvertible $x$ in $R$ with $\mathrm{FPdim}(R)=55$ is
\begin{equation}
\left[\begin{array}{c|cc|cc|cc}
0 & 0 & 0 & 0 & 0 & 0 & 1 \\\hline
0 & 0 & 0 & 0 & 0 & 0 & 1 \\
0 & 0 & 0 & 0 & 0 & 0 & 1 \\\hline
0 & 0 & 0 & 0 & 0 & 0 & 1 \\
0 & 0 & 0 & 0 & 0 & 0 & 1 \\\hline
1 & 1 & 1 & 1 & 1 & 2 & 2 \\
0 & 0 & 0 & 0 & 0 & 3 & 2
\end{array}\right].
\end{equation}
These are the fusion rules of the character ring of the unique nonabelian group of order $55$ up to isomorphism, $C_{11}\rtimes C_5$.

\par In the case $\mathrm{FPdim}(R)=119$, the basis elements of Frobenius-Perron dimensions $1$ and $3$ form a fusion subring isomorphic to $R_{C_7\rtimes C_3}$.  Indeed the only nonnegative integer solutions to $3^2=n_1(1)+n_2(3)+n_3(7)$ where $n_1\in\{0,3\}$ have $n_3=0$.  If $x,y\in B$ have $\mathrm{FPdim}(x)=3$ and $\mathrm{FPdim}(y)=7$, then the only unknown fusion coefficient for $x$ is $c_{x,y}^y=c_{x,y^\ast}^{y^\ast}$ as by dimension comparison, $c_{x,y}^{y^\ast}=c_{x,y^\ast}^y=3-c_{x,y}^y$.   Meanwhile, measuring the dimension of $yy^\ast$ and noting $c_{y,y^\ast}^{x}=c_{y,y^\ast}^{x^\ast}$ and $c_{y,y^\ast}^{y}=c_{y,y^\ast}^{y^\ast}$, the only solution for $7^2=3+2c_{y,y^\ast}^{x}(3)+2c_{y,y^\ast}^y(7)$ is $c_{y,y^\ast}^{x}=3$ and $c_{y,y^\ast}^y=2$ which determines the fusion of $x$ since $3=c_{y,y^\ast}^{x}=c_{x^\ast,y^\ast}^{y^\ast}=c_{x,y}^{y}$.  By dimension comparison, $c_{y,y}^{y^\ast}=5$ and all other fusion coefficients for $y$ are zero except the ones involving invertible basis elements, giving the following fusion for $x$ and $y$.
\begin{equation}\label{matrices}
\left[\begin{array}{c|cc|cc|cc}
0 & 0 & 0 & 0 & 1 & 0 & 0 \\\hline
0 & 0 & 0 & 0 & 1 & 0 & 0 \\
0 & 0 & 0 & 0 & 1 & 0 & 0 \\\hline
1 & 1 & 1 & 1 & 1 & 0 & 0 \\
0 & 0 & 0 & 2 & 1 & 0 & 0 \\\hline
0 & 0 & 0 & 0 & 0 & 3 & 0 \\
0 & 0 & 0 & 0 & 0 & 0 & 3
\end{array}\right]
\qquad
\left[\begin{array}{c|cc|cc|cc}
0 & 0 & 0 & 0 & 0 & 0 & 1 \\\hline
0 & 0 & 0 & 0 & 0 & 0 & 1 \\
0 & 0 & 0 & 0 & 0 & 0 & 1 \\\hline
0 & 0 & 0 & 0 & 0 & 0 & 3 \\
0 & 0 & 0 & 0 & 0 & 0 & 3 \\\hline
1 & 1 & 1 & 3 & 3 & 2 & 2 \\
0 & 0 & 0 & 0 & 0 & 5 & 2
\end{array}\right]
\end{equation}
The fusion in (\ref{matrices}) along with their duals fail to be associative, hence there is no fusion ring of this type.


\begin{figure}[H]
\centering
\begin{equation*}
\begin{array}{|cccccc|}
\hline \mathrm{Rank} & \mathrm{Ring} & \mathrm{Prime} & \mathrm{Orbits} & \mathrm{FPdim} & \text{Basis }\mathrm{FPdims}\\\hline\hline
1 & - &- & -& - & - \\\hline
2 & - &- & -& - & - \\\hline
3 & \mathbb{Z}C_3 & 2& 2& 3 & 1,1,1 \\\hline
4 & \mathbb{Z}C_2^2 & 3&2& 4 & 1,1,1,1 \\\hline
5 & \mathbb{Z}C_5 & 2&3& 5 & 1,1,1,1,1 \\
   & R_{C_7\rtimes C_3} & 2&3& 21 & 1,1,1,3,3  \\\hline
6 & - & -& -& - &-  \\\hline
7 & \mathbb{Z}C_7 & 2,3&4,3 &7 & 1,1,1,1,1,1,1\\
  & S_2  & 3& 3& 16 &  1,1,1,1,2,2,2\\
  & R_{C_{13}\rtimes C_3} & 2&4& 39 & 1,1,1,3,3,3,3 \\
  & S_4 & 3& 3& 52 & 1,1,1,1,4,4,4\\
  & R_{C_{11}\rtimes C_5} & 2&4& 55 & 1,1,1,1,1,5,5\\\hline
8 & \mathbb{Z}C_2^3 & 7&2& 8 & 1,1,1,1,1,1,1,1  \\\hline
\end{array}
\end{equation*}
    \caption{Fusion rings $R$ w.\ fixed-point-free automorphisms of prime order, integer formal codegrees, and $\mathrm{rank}(R)<9$}%
    \label{fig:A}%
\end{figure}


\section{Fusion categories}

\par The primary interest in fusion rings is their realization as Grothendieck rings of \emph{fusion categories}, which have profound applications both representation theory and mathematical physics.  A fusion category is a fusion ring along with solutions to coherence conditions on the associative relations of the ring, typically referred to as \emph{pentagon equations} \cite[Section 2.1]{tcat}.  When the fusion coefficients of a fusion ring are larger than $1$ or the rank is large, the complexity of these equations becomes untenable and one is forced to develop other tools to determine whether such a ring is \emph{categorifiable}, i.e.\ realizable as the Grothendieck ring of a fusion category.  For example, all categorifications of the integral group rings $\mathbb{Z}G$ for a finite group $G$ are of the form $\mathrm{Vec}_G^\omega$ for some $\omega\in H^3(G,\mathbb{C}^\times)$, but if there exists a unique noninvertible basis element, this question is unanswered after two decades of research (e.g. \cite{MR3167494}\cite{MR3635673}\cite{MR3229513}\cite{MR1997336}\cite{schopieray2022categorification} and references within).  In the remainder of this section we will refer the reader to the standard textbook \cite{tcat} for definitions and basic results in this subject.

\subsection{The categorifications of small rank}  Here we will determine which of the fusion rings in Figure \ref{fig:A} are categorifiable.  This is equivalent to classifying all fusion categories with fixed-point-free fusion automorphisms of prime order and rank less than $9$, since the formal codegrees of the Grothendieck ring of such a category must be integers by Theorem \ref{thm1}(1) and \cite[Corollary 2.15]{ost15}.  If a fusion ring $R$ in Figure \ref{fig:A} is an integral group ring, its categorifications have been discussed \emph{ad nauseam} and are the categories of finite-dimensional $G$-graded vector spaces $\mathrm{Vec}_G^\omega$ twisted by a $3$-cocycle $\omega$.  If a fusion ring $R\neq R_\mathrm{pt}$ in Figure \ref{fig:A} has Frobenius-Perron dimension $pq$ for distinct primes $pq$, then there is a unique categorification by the representation category of the unique nonabelian finite group of order $pq$ \cite[Theorem 6.3]{MR2098028}.  The only fusion rings in Figure \ref{fig:A} not of these forms are $S_2$ and $S_4$.

\par We will demonstrate that neither $S_2$ or $S_4$ are categorifiable by a fusion category $\mathcal{C}$ in the usual manner of showing their \emph{double} $\mathcal{Z}(\mathcal{C})$, a modular fusion category in the sense of \cite[Section 8.14]{tcat}, cannot exist.  The modular fusion category $\mathcal{Z}(\mathcal{C})$ is related to the fusion category $\mathcal{C}$ by the restriction tensor functor $F:\mathcal{Z}(\mathcal{C})\to\mathcal{C}$ and its adjoint $I:\mathcal{C}\to\mathcal{Z}(\mathcal{C})$ \cite[Section 9.2]{tcat}.  The formal codegrees are of particular importance  in such arguments as they control the decomposition of the induction of the tensor unit $I(\mathbbm{1}_\mathcal{C})\in\mathcal{Z}(\mathcal{C})$ into simple summands \cite[Theorem 2.13]{ost15}.  Note that to use this result, a \emph{spherical} structure is required.  But such a structure is innate for any fusion category with $\mathrm{FPdim}(\mathcal{C})\in\mathbb{Z}$ \cite[Corollary 9.6.6]{tcat} because such fusion categories are \emph{pseudounitary} \cite[Definition 9.4.4]{tcat}.  In this case we have no reason to discuss \emph{categorical} dimensions as they may be replaced with Frobenius-Perron dimensions without any loss of generality.

\begin{proposition}
There is no categorification of the fusion ring $S_2$.
\end{proposition}

\begin{proof}
Let $\mathcal{C}$ be a categorification of $S_2$, which must be pseudounitary and moreover spherical \cite[Corollary 9.6.6]{tcat}, whose simple objects we label \begin{equation}
g_0:=\mathbbm{1}_\mathcal{C},g_1,g_2,g_3,X_1,X_2,X_3
\end{equation}
with $X_j$ non-invertible and the invertible simple objects forming the group $C_2^2$.  The formal codegrees of $S_2$ are $16,16,16,16,4,4,4$, hence the simple summands of $I(g_0)$ have dimensions $1,1,1,1,4,4,4$ by \cite[Theorem 2.13]{ost15} and all have trivial ribbon twist \cite[Proposition 3.11]{schopieray2022categorification}.  Using \cite[Proposition 9.2.2]{tcat}, we compute that
\begin{equation}\label{ekwa0}
F(I(g_j))=7g_j+3\sum_{k\neq j}g_k.
\end{equation}
Let $Z\subset I(g_0)$ satisfy $\mathrm{FPdim}(Z)=4$.  The adjoint relation of $F$ and $I$ gives 
\begin{equation}
\dim_\mathbb{C}\mathrm{Hom}(I(g_0),Z)=\dim_\mathbb{C}\mathrm{Hom}(g_j,F(Z))=1
\end{equation}
by \cite[Theorem 2.13]{ost15} since $S_2$ is a commutative fusion ring.  Label the three simple summands $Z_1,Z_2,Z_3\subset I(g_0)$ with $\mathrm{FPdim}(Z_j)=4$ for $1\leq j\leq3$.  Note that if $F(Z_j)=g_0+3g_k$ for some $k\neq0$, then
\begin{equation}
\dim_\mathbb{C}(Z_j,I(g_k))=\dim_\mathbb{C}(F(Z_j),g_k)=3,
\end{equation}
thus $\dim_\mathbb{C}\mathrm{Hom}(g_j,I(g_j))\geq9$, violating Equation (\ref{ekwa0}).  Hence either $2Z_k\oplus Z_\ell\subset I(g_j)$ or $Z_1\oplus Z_2\oplus Z_3\subset I(g_j)$ for some $k\neq\ell\in\{1,2,3\}$ so that Equation (\ref{ekwa0}) is satisfied.
But since $\theta_{Z_j}=1$ for $j\in\{1,2,3\}$, in either case by \cite[Theorem 2.5]{ost15},
\begin{equation}
0=\mathrm{Tr}(\theta_{I(g_j)})=12+\sum_k\theta_{Y_k}\mathrm{FPdim}(Y_k)
\end{equation}
where $Y_k$ are the unknown simple summands of $I(g_j)$ and $\theta_{Y_k}$ their ribbon twists which are roots of unity.  This would subsequently imply $12\leq\sum_k|\mathrm{FPdim}(Y_k)|=4$ by the triangle inequality as $\mathrm{FPdim}(I(g_j))=16$.  So no categorification of $S_2$ can exist.
\end{proof}

\begin{note}
One could also note that $S_2$ is a $C_2^2$-graded extension of $\mathbb{Z}C_2^2$.  Thus any categorification of $S_2$ is subject to the extension theory of \cite[Theorem 1.3]{MR2677836}.  The choice to not pursue this option was based on the inevitable need for the induction-restriction functors in the proof of Proposition \ref{s4}, and the length of the argument not significantly decreasing either way.
\end{note}

\par A description of integral fusion categories of Frobenius-Perron dimension $pq^2$ for distinct primes $p$ and $q$ was given in \cite{MR2511638}.  One slice of this description are categories which are \emph{group-theoretical}, i.e.\ those whose doubles coincide with $\mathcal{Z}(\mathrm{Vec}_G^\omega)$ for a finite group $G$ and $3$-cocycle $\omega$.  It is still a nontrivial task to determine how such a given category might be realized by group-theoretical data, but for the fusion ring $S_4$ with Frobenius-Perron dimension $52=2^2\cdot13$, the options are limited.

\begin{proposition}\label{s4}
There is no categorification of the fusion ring $S_4$.
\end{proposition}

\begin{proof}
Let $\mathcal{C}$ be a categorification of $S_4$, which must be pseudounitary and moreover spherical \cite[Corollary 9.6.6]{tcat}.  Even more, $\mathcal{C}$ must be group-theoretical by \cite[Theorem 1.1]{MR2511638}, i.e.\ the double $\mathcal{Z}(\mathcal{C})$ must be braided equivalent to $\mathcal{Z}(\mathrm{Vec}_G^\omega)$ for some finite group $G$ of order $52$ and $3$-cocycle $\omega\in H^3(G,\mathbb{C}^\times)$.  We will first determine there is only one possible such finite group $G$.

\par The formal codegrees of $S_4$ are $52,13,13,13,4,4,4$, hence the simple summands of the induction $I(\mathbbm{1}_\mathcal{C})$ have dimensions $1,4,4,4,13,13,13$ by \cite[Theorem 2.13]{ost15}.  As the forgetful functor $F:\mathcal{Z}(\mathcal{C})\to\mathcal{C}$ is a tensor functor, if $X\in\mathcal{O}(\mathcal{Z}(\mathcal{C})_\mathrm{pt})$, then $\mathbbm{1}_\mathcal{C}\cong F(X)^{\otimes2}\cong F(X^{\otimes2})$.  Thus $X^{\otimes2}\subset I(\mathbbm{1}_\mathcal{C})$ and $X^{\otimes2}$ is invertible.  The unit $\mathbbm{1}_{\mathcal{Z}(\mathcal{C})}$ is the only such simple object.  Moreover all invertible objects in the double of $\mathcal{C}$ have order at most 2.  Now $\mathcal{Z}(\mathrm{Vec}_G^\omega)$ contains a fusion subcategory equivalent to $\mathrm{Rep}(G)$ (refer to \cite{MR2552301}, for example), so the only possibility for a finite group $G$ of order $52$ such that $\mathcal{Z}(\mathcal{C})\simeq\mathcal{Z}(\mathrm{Vec}_G^\omega)$ is a braided equivalence is the dihedral group $D_{26}$ as for $G\not\cong D_{26}$ with $|G|=52$, $G$ has $1$-dimensional representations of order $4$ or greater.  Analysis of these twisted doubles is possible by inspection, using the methods and formulas found in \cite{MR4257620}, for example.  As $13$ is a formal codegree of $\mathcal{C}$, there must exist a simple summand $X\subset I(\mathbbm{1}_\mathcal{C})$ with $\mathrm{FPdim}(X)=4$.  None of the simple objects of $\mathcal{Z}(\mathrm{Vec}_{D_{26}}^\omega)$ have dimension 4, so no such fusion category $\mathcal{C}$ exists.
\end{proof}


\subsection{Fusion automorphisms and tensor autoequivalences}

Recall that a \emph{tensor autoequivalence} $F:\mathcal{C}\to\mathcal{C}$ of a fusion category $\mathcal{C}$ is an exact and faithful $\mathbb{C}$-linear functor along with functorial isomorphisms $J_{X,Y}:F(X)\otimes F(Y)\to F(X\otimes Y)$ for all $X,Y\in\mathcal{C}$ such that $F(\mathbbm{1}_\mathcal{C})\cong\mathbbm{1}_\mathcal{C}$ and the pair $F,J$ satifies coherence conditions \cite[Definition 2.4.1]{tcat}.  Needless to say, the disparity between automorphisms of Grothendieck rings and tensor autoequivalences is stark.  As this exposition is primarily about the former, we only illustrate this disparity in one example for a fusion ring of small rank from Figure \ref{fig:A} which is not an integral group ring, since these autoequivalences are discussed in detail in \cite[Section 2.6]{tcat}.

\par The key observation is that if there exists a fixed-point-free fusion automorphism $\phi$ of the Grothendieck ring of a fusion category $\mathcal{C}$ of prime order $p\in\mathbb{Z}_{\geq2}$ which lifts to a tensor autoequivalence of $\mathcal{C}$, this is equivalent to a \emph{categorical action} of $C_p$ on $\mathcal{C}$ \cite[Definition 4.15.1]{tcat} and one may consider the $C_p$-\emph{equivariantization} of $\mathcal{C}$, denoted $\mathcal{C}^{C_p}$ \cite[Definition 2.7.2]{tcat}.  The simple objects of the fusion category $\mathcal{C}^{C_p}$ are the $\phi$-orbits $\tilde{X}$ of the simple objects $X\in\mathcal{C}$ with $\mathrm{FPdim}(\tilde{X})=p\mathrm{FPdim}(X)$ and invertible objects $g$ for $g\in C_p$ \cite[Proposition 4.15.9]{tcat}. 

\begin{proposition}
The duality fusion automorphism of $R_{C_7\rtimes C_3}$ does not lift to a tensor autoequivalence of $\mathrm{Rep}(C_7\rtimes C_3)$.
\end{proposition}

\begin{proof}
Assume to the contrary that duality is a tensor autoequivalence of $\mathrm{Rep}(C_7\rtimes C_3)$.  The simple objects of $\mathcal{C}:=\mathrm{Rep}(C_7\rtimes C_3)^{C_2}$, which must be pseudounitary and moreover spherical \cite[Corollary 9.6.6]{tcat},  are $e,g,X,Y$ where $\mathrm{FPdim}(X)=2$ and $\mathrm{FPdim}(Y)=6$.  Moreover $e,g,X$ form a fusion subcategory of $\mathcal{C}$ with $\mathrm{Rep}(S_3)$ fusion rules.  The fusion rules of $\mathcal{C}$ are then
\begin{align}
\left[\begin{array}{cccc}
1 & 0 & 0 & 0 \\
0 & 1 & 0 & 0 \\
0 & 0 & 1 & 0 \\
0 & 0 & 0 & 1
\end{array}
\right]\qquad\left[\begin{array}{cccc}
0 & 1 & 0 & 0 \\
1 & 0 & 0 & 0 \\
0 & 0 & 1 & 0 \\
0 & 0 & 0 & 1
\end{array}
\right] \\
\left[\begin{array}{cccc}
0 & 0 & 1 & 0 \\
0 & 0 & 1 & 0 \\
1 & 1 & 1 & 0 \\
0 & 0 & 0 & 2
\end{array}
\right]\qquad\left[\begin{array}{cccc}
0 & 0 & 0 & 1 \\
0 & 0 & 0 & 1 \\
0 & 0 & 0 & 2 \\
1 & 1 & 2 & 5
\end{array}
\right]
\end{align}
with formal codegrees $42,7,3,2$.  Therefore the simple summands $\mathbbm{1}_{\mathcal{Z}(\mathcal{C})}$, $Z_1$, $Z_2$, $Z_3$ of $I(e)$ have dimensions $1$, $6$, $14$ and $21$, respectively.  We compute
\begin{equation}\label{eqy}
F(I(e))=4e\oplus2g\oplus3X\oplus5Y.
\end{equation}
As $\mathrm{FPdim}(Y)=6$ all the summands $Y$ appearing in Equation (\ref{eqy}) must be summands of $F(Z_2)$ and $F(Z_3)$.  This forces the decompositions $F(Z_1)=e\oplus g\oplus2X$, $F(Z_2)=e\oplus g\oplus2Y$, and $F(Z_3)=e\oplus X\oplus3Y$.  Similarly
\begin{equation}
F(I(g)-Z_1-Z_2)=2g\oplus X\oplus3Y.
\end{equation}
Therefore $I(g)$ has two other nonisomorphic simple summands $Z_4$, $Z_5$ whose dimensions must be among the collection $1$, $3$, $7$, $21$ as their squares must divide $\mathrm{FPdim}(\mathcal{Z}(\mathcal{C}))=2^2\cdot3^2\cdot7^2$ \cite[Proposition 8.14.6]{tcat} and they must be odd since they contain $g$ in their forgetful image with multiplicity $1$.  The only possibility is that $F(Z_4)=g$ and $F(Z_5)=g\oplus X\oplus3Y$ without loss of generality.  Therefore $\mathcal{Z}(\mathcal{C})_\mathrm{pt}$ is rank $2$.  Now note that $\mathcal{C}$ is group-theoretical since its Frobenius-Perron dimension is $2\cdot3\cdot7$ \cite[Theorem 9.2]{MR2735754}, i.e.\ $\mathcal{Z}(\mathcal{C})$ is braided equivalent to $\mathcal{Z}(\mathrm{Vec}_G^\omega)$ for a finite group $G$ of order $42$ which can have at most $2$ isomorphism classes of $1$-dimensional representations, as these correspond to distinct invertible objects in $\mathrm{Rep}(G)\subset\mathcal{Z}(\mathrm{Vec}_G^\omega)$.  Therefore $G\cong D_{21}$.  Analysis of these twisted doubles is possible by inspection, using the methods and formulas found in \cite{MR4257620}, for example.  We find $\mathrm{FPdim}(X)=2$ for all simple noninvertible $X\in\mathcal{Z}(\mathrm{Vec}_{D_{21}}^\omega)$ over all $\omega\in H^3(D_{21},\mathbb{C}^\times)$, and so we conclude no such category $\mathcal{C}$ exists.
\end{proof}

\bibliographystyle{plain} 
\bibliography{bib}

\end{document}